\documentclass[a4paper, 11pt,reqno]{amsart}
\usepackage[english]{babel}
\usepackage[utf8]{inputenc}
\usepackage[T1]{fontenc}
\usepackage{amsthm, amsmath, amssymb, amsrefs, enumerate, enumitem, mathtools, mathrsfs, lmodern, microtype}
\usepackage{a4wide}
\usepackage{hyperref}

\theoremstyle{plain}
\newtheorem{thm}{Theorem}[section]
\newtheorem*{thm*}{Theorem}
\newtheorem{prop}[thm]{Proposition}
\newtheorem{lemma}[thm]{Lemma}
\newtheorem{cor}[thm]{Corollary}
\theoremstyle{definition}
\newtheorem{defn}[thm]{Definition}

\theoremstyle{remark}
\newtheorem*{rmk}{Remark}
\newtheorem*{rmks}{Remarks}

\newcommand{\C}{\mathbb{C}}
\renewcommand{\H}{\mathbb{H}}
\newcommand{\Z}{\mathbb{Z}}
\newcommand{\Q}{\mathbb{Q}}
\newcommand{\N}{\mathbb{N}}
\newcommand{\R}{\mathbb{R}}
\newcommand{\slz}{{\text {\rm SL}}_2(\mathbb{Z})}
\newcommand{\re}{\textnormal{Re}}
\newcommand{\im}{\textnormal{Im}}
\newcommand{\vt}[1]{\left\lvert #1 \right\rvert}
\newcommand{\reg}{\textnormal{reg}}
\newcommand{\Ec}{\mathcal{E}}
\newcommand{\Fc}{\mathcal{F}}
\newcommand{\Qc}{\mathcal{Q}}

\DeclareMathOperator{\tr}{tr}
\newcommand{\dm}{\mathrm{d}}
\newcommand{\Hs}{\mathscr{H}}
\newcommand{\Ks}{\mathscr{K}}
\newcommand{\af}{\mathfrak{a}}
\newcommand{\ef}{\mathfrak{e}}
\newcommand{\id}{\mathrm{id}}

\makeatletter
\let\@@pmod\pmod
\DeclareRobustCommand{\pmod}{\@ifstar\@pmods\@@pmod}
\def\@pmods#1{\mkern4mu({\operator@font mod}\mkern 6mu#1)}
\makeatother

\setlist{nosep}
\setlist{noitemsep}
\numberwithin{equation}{section}

\binoppenalty=\maxdimen
\relpenalty=\maxdimen

\title{A modular framework for generalized Hurwitz class numbers III}
\author{Andreas Mono}
\address{Department of Mathematics, Vanderbilt University, 1326 Stevenson Center, Nashville TN 37240, USA}
\email{andreas.mono@vanderbilt.edu}

\begin{document}

\begin{abstract}
In $2003$, Pei and Wang introduced higher level analogs of the classical Cohen--Eisenstein series. In recent joint work with Beckwith, we found a weight $\frac{1}{2}$ sesquiharmonic preimage of their weight $\frac{3}{2}$ Eisenstein series under $\xi_{\frac{1}{2}}$ utilizing a construction from seminal work by Duke, Imamo\={g}lu and T\'{o}th. In further joint work with Beckwith, when restricting to prime level, we realized our preimage as a regularized Siegel theta lift and evaluated its (regularized) Fourier coefficients explicitly. This relied crucially on work by Bruinier, Funke and Imamo\={g}lu. In this paper, we extend both works to higher weights. That is, we provide a harmonic preimage of Pei and Wang's generalized Cohen--Eisenstein series under $\xi_{\frac{3}{2}-k}$, where $k > 1$. Furthermore, when restricting to prime level, we realize them as outputs of a regularized Shintani theta lift of a higher level holomorphic Eisenstein series, which builds on recent work by Alfes and Schwagenscheidt. Lastly, we evaluate the regularized Millson theta lift of a higher level Maass--Eisenstein series, which is known to be connected to the Shintani theta lift by a differential equation by earlier work of Alfes and Schwagenscheidt.
\end{abstract}




\maketitle

\section{Introduction and statement of results}

\subsection{Overall motivation}
Throughout, we let $k >1$ be an integer. In $1975$, Cohen \cite{cohen75} constructed weight $k+\frac{1}{2}$ Eisenstein series 
\begin{align} \label{eq:Cohendef}
\Hs_{k}(\tau) \coloneqq \sum_{n \geq 0} H(k,n) q^n, \qquad q \coloneqq e^{2\pi i \tau}, \qquad \tau \in \H \coloneqq \left\{\tau = u+iv \colon v > 0\right\},
\end{align}
of level $4$ whose Fourier coefficients $H(k,n)$ are values of Dirichlet $L$-functions. Among (powers of) Jacobi's $\theta$-function or Dedekind's $\eta$-function, these are some of the earliest examples of half integral weight modular forms and have become ubiquitous in the theory of modular forms. In weight $\frac{3}{2}$, Cohen's coefficients $H(1,n)$ become the classical Hurwitz class numbers $H(n)$ by Dirichlet's class number formula. However, their generating function $\Hs_{1}$ is no longer modular.

This was resolved by Zagier \cites{zagier75-2, zagier76-2} around the same time. He discovered a non-holomorphic modular completion $\mathcal{H}$ of $\Hs_{1}$ having weight $\frac{3}{2}$ and level $4$. In today's terminology, he constructed a so-called harmonic Maass form by applying ``Hecke's convergence trick'' to a weight $\frac{3}{2}$ Eisenstein series projected to Kohnen's plus space. Roughly speaking, harmonic Maass forms are real-analytic analogs of (complex-analytic) modular forms along with a relaxed growth condition towards the cusps (see Definition \ref{defn:hmfdef} for a precise definition). The theory of harmonic Maass forms goes back to Bruinier and Funke \cite{brufu02} and is built around the properties of their \emph{$\xi_{\kappa}$-operator} ($\kappa \in \frac{1}{2}\Z$)
\begin{align} \label{eq:xidef}
\xi_{\kappa}f \coloneqq 2i v^{\kappa} \overline{\frac{\partial f}{\partial\overline{\tau}}}.
\end{align}
Although the $\xi_{\kappa}$-operator maps the space of weight $\kappa$ harmonic Maass forms surjectively onto the space of weakly holomorphic modular forms of weight $2-\kappa$, finding explicit preimages under $\xi_{\kappa}$ is an open problem in general. It is solved in the case of Eisenstein- and Poincar{\'e} series
\cites{BO1, BO2, wag, rhowal, mamoro1, meonro, mo1} and weight 2 newforms \cites{afmr, agor} (see \eqref{eq:thetaliftconnection} below) so far.

\subsection{Previous results}

Let $N > 1$ be odd and square-free throughout, and let $\ell \mid N$. In $2003$, Pei and Wang \cite{peiwang} generalized Cohen's work in the level aspect by constructing the \emph{generalized Cohen--Eisenstein series}
\begin{align} \label{eq:Hsdef}
\Hs_{k,\ell,N}(\tau) \coloneqq \sum_{n \geq 0} H_{k,\ell,N}(n) q^n,
\end{align}
which resembles equation \eqref{eq:Cohendef}. The functions $\Hs_{k,\ell,N}$ arise as the plus space projections of (modifications of) Eisenstein series investigated by Pei in \cites{peiI, peiII} in weight $\frac{3}{2}$. Pei and Wang showed that the functions $\Hs_{k,\ell,N}$ are modular forms of weight $k+\frac{1}{2}$ and level $4N$ for every $\ell \mid N$. In weight $\frac{3}{2}$, the level structure allows some refined inspection of the (obstruction towards) modularity of $\Hs_{1,\ell,N}$. More precisely, it turns out that $\Hs_{1,1,N}$ continues to be non-modular, while $\Hs_{1,\ell,N}$ is modular for $1 < \ell \mid N$. Pei and Wang related $\Hs_{k,\ell,N}$ to standard Eisenstein series via the Shimura lift, which we summarize in Subsection \ref{subsec:peiwangprelim}. More details can be found in \cite{peiwangbook}

Due to their technical nature, we define the numbers $H_{k,\ell,N}(n)$ in equations \eqref{eq:HNNdef}, \eqref{eq:HellNdef} below. In the spirit of Cohen's definition of the numbers $H(k,n)$, Pei and Wang's definition of $H_{k,\ell,N}(n)$ can be viewed as a suitable modification of Dirichlet's class number formula. However, compared to Cohen's numbers $H(k,n)$, the numbers $H_{k,\ell,N}$ involve some additional local factors pertaining to $\ell$ and $N$. We have $H(k,n) = H_{k,1,1}(n)$ as well as $\Hs_{k} = \Hs_{k,1,1}$.

In part I of this series joint with Beckwith \cite{bemo1}, we proved that a certain linear combination of $\Hs_{1,1,1}$ (the Hurwitz class number generating function) and $\Hs_{1,1,N}$ yields a holomorphic modular form of weight $\frac{3}{2}$ and level $4N$. Moreover, we found a higher level analog of Zagier's non-holomorphic Eisenstein series $\mathcal{H}$ completing $\Hs_{1,1,N}$ to a harmonic Maass form of weight $\frac{3}{2}$ and level $4N$. Both results relied on a careful inspection of a weight $\frac{1}{2}$ Maass--Eisenstein series around its spectral point, which extends a construction from seminal work of Duke, Imamo\={g}lu and T\'{o}th \cite{dit11annals} to higher levels. Their results were reinterpreted in the framework of Siegel theta lifts in foundational work of Bruinier, Funke and Imamo\={g}lu \cite{brufuim}. In part II of this series joint with Beckwith \cite{bemo2}, we restricted to odd prime level to connect our results from the first part to Bruinier, Funke and Imamo\={g}lu's regularized Siegel theta lift. In turn, we obtained an explicit evaluation of the (regularized) quadratic traces appearing as Fourier coefficients of their theta lift in the Eisenstein case. For instance, we showed in \cite{bemo2}*{Theorem 1.4, 6.1 (i)} that (see equation \eqref{eq:Quadraticformssets} for the definition of $\mathcal{Q}_{p,n}$)
\begin{align*}
\sum_{Q \in \mathcal{Q}_{p,n} \slash \Gamma_0(p)} \frac{1}{\vt{\Gamma_0(p)_Q}} &= \frac{4(p+1)}{p} H_{1,1,p}(-n) - \frac{2(p+1)}{p-1} H_{1,p,p}(-n) \\
&= 4H_{1,1,1}(-n) - 2H_{1,p,p}(-n)
\end{align*}
whenever $n < 0$ with $n \equiv 0,1 \pmod*{4}$. The second equality follows by \cite{bemo1}*{Theorem 1.1}. Here, we assumed that $k = 1$ (i.e.\ weight $\frac{3}{2}$). If $n > 0$, class numbers can be studied as cycle integrals of the constant function $1$, see Siegel \cite{siegel}*{pp.\ 86 or pp.\ 116} for the case of level $N=1$. Corollary \ref{cor:shintanicor} below extends this perspective to $k > 1$ and odd prime levels. For these reasons, we refer to the numbers $H_{k,\ell,N}(n)$ as generalized Hurwitz class numbers.

\subsection{Main results}

Adapting Zagier's computations on $\mathcal{H}$ to other weights, Wagner \cite{wag} constructed preimages of the classical Cohen--Eisenstein series $\Hs_{k}$ under $\xi_{\frac{3}{2}-k}$, which are harmonic Maass forms of weight $\frac{3}{2}-k$ and level $4$. Let $\cdot\vert_{\kappa}\gamma$ be the Petersson slash operator (defined in equation \eqref{eq:slashdef}), $\mathrm{pr}^{+}$ be the projection operator to Kohnen's plus space (see \cite{koh85}*{Proposition 3}, \cite{thebook}*{Proposition 6.7}) and $\Gamma_{\infty} \coloneqq \left\{\pm \left(\begin{smallmatrix} 1 & n \\ 0 & 1 \end{smallmatrix}\right)  \colon n \in \Z\right\}$. We define the \emph{Maass--Eisenstein series of weight $\frac{3}{2}-k$}
\begin{align} \label{eq:Fcpdef}
\Fc_{\frac{3}{2}-k,4N}^{+}(\tau) \coloneqq \mathrm{pr}^{+} \sum_{\gamma \in \Gamma_{\infty} \backslash \Gamma_0(4N)} v^{k-\frac{1}{2}} \Big\vert_{\frac{3}{2}-k}\gamma.
\end{align}
The following higher level analog of Wagner's result extends our earlier work joint with Beckwith \cite{bemo1} to higher weights.
\begin{thm} \label{thm:CEpreimage}
Let $k > 1$, $N \in \N$ be odd and square-free. Let $\zeta$ be the usual analytic continuation of the Riemann $\zeta$-function. Then we have
\begin{align*}
\xi_{\frac{3}{2}-k} \Fc_{\frac{3}{2}-k,4N}^{+}(\tau)
= \frac{2k-1}{3\zeta(1-2k)} \sum_{\ell \mid N} \Big(\prod_{\substack{p \text{ prime} \\ p\mid \ell}} \frac{1}{1-p^{2k-1}}\Big) \Big(\prod_{\substack{p \text{ prime} \\ p\mid \frac{N}{\ell}}} \frac{p-1}{p^{2k}-p}\Big) \Hs_{k,\ell,N}(\tau).
\end{align*}
\end{thm}

\begin{rmk}
The right hand side simplifies to \cite{wag}*{Theorem 1.1 (2)} if $k > 1$ and $N = 1$. If $k=1$ and $N > 1$ then the right hand side becomes the holomorphic part of \cite{bemo1}*{Theorem 1.3 (iii)}.
\end{rmk}

Since Niwa's reinterpretation \cite{niwa} of the classical Shimura map \cite{shim} as a theta lift (see Cipra \cite{cip83} as well), theta lifts became a central theme in inspecting the interplay between integral and half integral weight cusp forms \cites{koh80, koh82, koza81, koza84, grokoza, koh85, sz88, zag02}. An idea of Harvey and Moore \cite{hamo} as well as Borcherds \cite{bor98} to regularize the theta integral extends the domain and range of such theta lifts to spaces of harmonic Maass forms $H_{\kappa}^{!}(N)$ \cites{BO4, BEY, BO3, alfehl, alfes, BY, mamo1, mamo2, bor00, males, BS}. Building on work by Bruinier and Funke \cite{brufu06} as well as by Bruinier, Funke and Imamo\={g}lu \cite{brufuim}, Alfes and Schwagenscheidt \cites{alneschw18, alneschw21, alfesthesis, schw18} investigated various regularized theta lifts of harmonic Maass forms $f \in H_{2-2k}^{!}(N)$ and $g \in H_{2k}^{!}(N)$. In this paper, we utilize their regularized Shintani theta lift $I^{(\mathrm{S})}(g,\tau)$ (defined in equation \eqref{eq:ISdef}) as well as their regularized Millson theta lift $I^{(\mathrm{M})}(f,\tau)$ (defined in equation \eqref{eq:IMdef}). Then, Alfes and Schwagenscheidt proved that both regularized theta lifts satisfy the differential equation (see \cite{alneschw21}*{Proposition 5.5})
\begin{align} \label{eq:thetaliftconnection}
\xi_{\frac{3}{2}-k} I^{(\mathrm{M})}\left(f, \tau\right) = -\frac{1}{2\sqrt{N}} I^{(\mathrm{S})}\left(\xi_{2-2k}f, \tau\right).
\end{align}
Extending an idea of Guerzhoy \cites{gue14, gue15}, Alfes, Griffin, Ono and Rolen \cite{agor} used equation \eqref{eq:thetaliftconnection} for harmonic Maass forms with cuspidal image under $\xi_{\kappa}$ to construct weight $0$ harmonic preimages of weight $2$ newforms under $\xi_0$.

We apply equation \eqref{eq:thetaliftconnection} to the \emph{weight $2-2k$ Maass--Eisenstein series} 
\begin{align} \label{eq:Fcdef}
\Fc_{2-2k,N}(\tau) \coloneqq \sum_{\gamma \in \Gamma_{\infty} \backslash \Gamma_0(N)} v^{2k-1} \Big\vert_{2-2k}\gamma.
\end{align}
Its image under $\xi_{2-2k}$ is given by a non-zero multiple (see equation \eqref{eq:Eisensteinshadow}) of the \emph{weight $2k$ and level $N$ holomorphic Eisenstein series} 
\begin{align} \label{eq:Ecdef}
\Ec_{2k,N}(\tau) \coloneqq \sum_{\gamma \in \Gamma_{\infty} \backslash \Gamma_0(N)} 1 \big\vert_{2k}\gamma.
\end{align}
We begin by evaluating the right hand side in equation \eqref{eq:thetaliftconnection}. To this end, we recall that there is a ``trace map'' from vector-valued forms $F$ with components $F_h$ for the Weil representation to scalar-valued modular forms of higher level, see equation \eqref{eq:plusspaceprojection}. If $N=1$ or $N = p$ is a prime then the image satisfies the plus space condition (see \cite{brufuim}*{Example 2.2}) and we find the following result by utilizing this map.
\begin{thm} \label{thm:shitanilift}
Let $N = p$ be an odd prime and $k \in 2\N$. Then we have 
\begin{multline*}
\sum_{h \pmod*{2p}} I_h^{(\mathrm{S})}\left(\Ec_{2k,p}, 4p\tau\right) = (-1)^{\frac{k}{2}} \sqrt{p} \frac{\Gamma(k)\zeta(k)}{2^{k-1} \pi^k \zeta(1-2k)} \\
\times \left(p^{-k}\Hs_{k,1,1}(\tau) + \left(1-p^{-k}\right)\left(\frac{p-1}{p^{2k}-p} \Hs_{k,1,p}(\tau) + \frac{\Hs_{k,p,p}(\tau)}{1-p^{2k-1}}\right)\right).
\end{multline*}
\end{thm}

\begin{rmks}
\
\begin{enumerate}
\item Bringmann, Guerzhoy and Kane \cites{bgk14, bgk15} extended the classical Shintani lift \cite{shin} to weakly holomorphic modular forms, and evaluated it on Poincar{\'e} series for $\slz$, see \cite{bgk14}*{Lemma 4.1}.
\item It is well-known that the Shintani lift of $\Ec_{2k} \coloneqq \Ec_{2k,1}$ evaluates to $\Hs_{k}$ (the classical Cohen--Eisenstein series from equation \eqref{eq:Cohendef}), see \cite{koza84}*{p.\ 240--241}, \cite{fumi}*{Theorem 9.2}, \cite{bgk14}*{Theorem 1.3, Lemma 4.1}.
\end{enumerate}
\end{rmks}

In particular, the regularized real quadratic traces (see Subsubsection \ref{subsubsec:quadratictraces}) of $\Ec_{2k,p}$ are given by the corresponding Fourier coefficients of the right hand side in Theorem \ref{thm:shitanilift}.
\begin{cor} \label{cor:shintanicor}
If $N = p$ is an odd prime, $k \in 2\N$, and $1 \leq D = \square$, then we have
\begin{multline*}
\sum_{Q \in \Qc_{p,D} \slash \Gamma_0(p)} \int_{\Gamma_0(p) \backslash S_Q}^{\reg} \Ec_{2k,p}(z) Q(z,1)^{k-1} \dm z = (-1)^{\frac{k}{2}} \frac{\Gamma(k)\zeta(k)}{2^{k-1} \pi^k \zeta(1-2k)} \\
\times \left(p^{-k} H_{k,1,1}(D) + \left(1-p^{-k}\right) \left(\frac{p-1}{p^{2k}-p} H_{k,1,p}(D) + \frac{H_{k,p,p}(D)}{1-p^{2k-1}} \right) \right).
\end{multline*}
\end{cor}

\begin{rmk}
An excellent overview of similar formulas can be found in further work by Duke, Imamo\={g}lu and T\'{o}th \cite{dit18}.
\end{rmk}

We move to the left-hand side in equation \eqref{eq:thetaliftconnection} and evaluate the regularized Millson lift of $\Fc_{2-2k,p}$ projected to Kohnen's plus space.
\begin{thm} \label{thm:millsonlift}
Let $N = p$ be an odd prime and $k \in 2\N$. Then we have
\begin{multline*}
\sum_{h \pmod*{2p}} I_h^{(\mathrm{M})}\left(\Fc_{2-2k,p}, 4p\tau\right) \\
= 3 (-1)^{\frac{k}{2}-1} \frac{\Gamma(k)\zeta(k)}{2^{k} \pi^k} (4p)^{k-\frac{1}{2}} \left(p^{-k} \Fc_{\frac{3}{2}-k,4}^{+}(\tau)  + \left(1-p^{-k}\right) \Fc_{\frac{3}{2}-k,4p}^{+}(\tau)\right).
\end{multline*}
\end{thm}

\begin{rmk}
Combining Theorems \ref{thm:CEpreimage}, \ref{thm:shitanilift} and \ref{thm:millsonlift} verifies equation \eqref{eq:thetaliftconnection} for the input $f = \Fc_{2-2k,p}$.
\end{rmk}

In particular, Theorem \ref{thm:millsonlift} yields an evaluation of the imaginary quadratic traces (see equation \eqref{eq:trimagdef}) of the \emph{weight $0$ and level $p$ non-holomorphic Eisenstein series}
\begin{align} \label{eq:Fc0def}
\Fc_{0,N}(\tau,k) \coloneqq \sum_{\gamma \in \Gamma_{\infty} \backslash \Gamma_0(N)} \im(\gamma\tau)^{k}.
\end{align}
To describe this, we let $\Ks_{\frac{3}{2}-k,4N}^{+}(m,n;s)$ be the Kloosterman plus zeta function appearing in the Fourier expansion of $\Fc_{\frac{3}{2}-k,4N}^{+}$ about $i\infty$, which is defined in equation \eqref{eq:Ksdef}.
\begin{cor} \label{cor:millsoncor}
If $N = p$ is an odd prime, $k \in 2\N$, and $D > 0$ such that $-D$ is a discriminant, then we have
\begin{multline*}
\sum_{Q \in \Qc_{p,-D} \slash \Gamma_0(p)} \frac{1}{\vt{\Gamma_0(p)_Q}} \Fc_{0,p}\left(\tau_Q,k\right) = (-1)^{\frac{k}{2}-1} 2^{k+\frac{1}{2}}i^{k-\frac{3}{2}} \zeta(k) D^{\frac{k}{2}} \\
\times \left(p^{-k} \Ks_{\frac{3}{2}-k,4}^{+}\left(0,D;k+\frac{1}{2}\right)  + \left(1-p^{-k}\right) \Ks_{\frac{3}{2}-k,4p}^{+}\left(0,D;k+\frac{1}{2}\right)\right).
\end{multline*}
\end{cor}

\begin{rmks}
\
\begin{enumerate}
\item This can be viewed as a level $p$ analog of the corresponding case in \cite{dit11annals}*{Proposition 4}.
\item We evaluate $\Ks_{\frac{3}{2}-k,4N}^{+}(0,n;s)$ explicitly in Lemma \ref{lem:shimurasturmresult}.
\end{enumerate}

\end{rmks}

\subsection{Outline of the paper}
The paper is organized as follows. We summarize the overall framework in Section \ref{sec:prelim}. Section \ref{sec:proofCE} is devoted to the proof of Theorem \ref{thm:CEpreimage}. The purpose of Section \ref{sec:liftfourier} is to work out the Fourier expansions of both $I^{(\mathrm{S})}\left(\Ec_{2k,p}, \tau\right)$ and $I^{(\mathrm{M})}\left(\Fc_{2-2k,p}, \tau\right)$ following Alfes and Schwagenscheidt. The content of Section \ref{sec:proofShintani} is an evaluation of the real quadratic traces of $\Ec_{2k,p}$ for non-square discriminants. This enables us to prove Theorem \ref{thm:shitanilift} as well as Corollary \ref{cor:shintanicor}. The subject of Section \ref{sec:proofMillson} is the proof of both Theorem \ref{thm:millsonlift} and Corollary \ref{cor:millsoncor} by employing Theorems \ref{thm:CEpreimage}, \ref{thm:shitanilift} and equation \eqref{eq:thetaliftconnection}. We conclude by collecting some questions for further work in Section \ref{sec:questions}.

\section*{Acknowledgements}
We would like to thank Scott Ahlgren, Michael Griffin and Martin Raum for helpful conversations as well as Olivia Beckwith, Kathrin Bringmann, Nikolaos Diamantis, Eleanor McSpirit, and Larry Rolen for valuable feedback on an earlier version of this paper. Moreover, we thank the anonymous referee for many helpful comments.

\section*{Notation}

\subsection*{General modular forms notation}
\begin{itemize}
\item $k > 1$ is an integer, $\kappa \in \frac{1}{2}\Z$ is an integer or a half integer,
\item $\H$ is the complex upper half plane,
\item $\tau = u+iv \in \H$, $z = x+iy \in \H$, $q = e^{2\pi i \tau}$, $\dm\mu(z) \coloneqq \frac{\dm x \dm y}{y^2}$,
\item $N \in \N$ is odd and square-free, 
\item $\ell$ is a divisor of $N$, 
\item $p$ is an odd prime, 
\item $\Gamma(s,z)$ is the incomplete $\Gamma$-function, $\Gamma(s)$ is the $\Gamma$-function, see equation \eqref{eq:Gammadef},
\item $\af \in \Q \cup \{i\infty\}$ is a cusp with scaling matrix $\sigma_{\af} \in \slz$,
\item $\Gamma_0(N)$ is the Hecke congruence subgroup of $\slz$,
\item $\Gamma_{\infty} = \left\{\pm \left(\begin{smallmatrix} 1 & n \\ 0 & 1 \end{smallmatrix}\right) \colon n \in \Z \right\} \leq \Gamma_0(N)$ is the stabilizer of the cusp $i\infty$,
\item $\left(\frac{a}{b}\right)$ is the Kronecker symbol,
\item $\chi_{d}(n) = \left(\frac{d}{n}\right)$ is a Dirichlet character,
\item $\id$ is the principal character of modulus $1$,
\item $L(s,\chi)$ is the (complete) Dirichlet $L$-function, 
\item $\zeta(s) = L(s,\id)$ is the (analytic continuation of the) Riemann zeta function,
\item $L_N(s,\chi)$ the incomplete Dirichlet $L$-function, see equation \eqref{eq:LNdef},
\item $\mu(n)$ the M{\"o}bius function,
\item $\sigma_{\ell,N,s}(n)$ and $\sigma_{N,s}(n)$ are certain modifications of the standard sum of powers of divisors function $\sigma_s$, see equation \eqref{eq:divisordef},
\item $M_{k}(N)$ is the $\C$-vector space of holomorphic modular forms of weight $k$ for $\Gamma_0(N)$,
\item $S_{k}(N)$ is the $\C$-vector space of cusp forms of weight $k$ for $\Gamma_0(N)$,
\item $M_{k}^{!}(N)$ is the $\C$-vector space of weight $k$ weakly holomorphic modular forms for $\Gamma_0(N)$,
\item $E_{k}^{+}(N) \subseteq M_{k}(N)$ is the Eisenstein subspace of forms satisfying the plus space condition,
\item $H_{\kappa}^{!}(N)$ is the $\C$-vector space of harmonic Maass forms of weight $\kappa$ for $\Gamma_0(N)$,
\item $\xi_k = 2iv^k\overline{\frac{\partial}{\partial\overline{\tau}}}$ is the $\xi_{\kappa}$-operator, see equation \eqref{eq:xidef},
\item $R_{\kappa}$ and $R_{\kappa}^n$ is the (iterated) Maass raising operator, see equation \eqref{eq:Maassraising},
\item $\Delta_k = -\xi_{2-k}\xi_k$ is the hyperbolic Laplace operator, see equation \eqref{eq:Deltasplitting},
\item $\cdot \vert_{k}\cdot$ resp.\ $\cdot \vert_{k, L}\cdot$ is the Petersson slash operator, see equations \eqref{eq:slashdef} and \eqref{eq:PeterssonVVdef},
\item $\cdot\vert\mathrm{pr}^{+}$ is the projection operator into Kohnen's plus space, see \cite{koh85}*{Proposition 3}, \cite{thebook}*{Proposition 6.7} for example,
\item $W_N = \left(\begin{smallmatrix} 0 & -1 \\ N & 1 \end{smallmatrix}\right)$ is the Fricke-involution.
\end{itemize}

\subsection*{Scalar-valued forms and their Fourier coefficients}
\begin{itemize}
\item $\Fc_{2-2k,N}(\tau)$ is the weight $2-2k$ level $4N$ Maass--Eisenstein series, see equation \eqref{eq:Fcdef},
\item $\Fc_{\frac{3}{2}-k,4N}^{+}(\tau)$ is the projection of the weight $\frac{3}{2}-k$ level $4N$ Maass--Eisenstein series to Kohnen's plus space, see equation \eqref{eq:Fcpdef},
\item $\Fc_{0,N}(\tau,s)$ is the weight $0$ level $N$ Maass--Eisenstein series with spectral parameter $s$, see equation \eqref{eq:Fc0def},
\item $\Ec_{2k,N}(\tau)$ is the weight $2k$ level $N$ holomorphic Eisenstein series, see equation \eqref{eq:Ecdef},
\item $\Hs_{k,\ell,N}(\tau)$ are the generalized Cohen--Eisenstein series, see equation \eqref{eq:Hsdef},
\item $H_{k,\ell,N}(n)$, are the generalized Hurwitz class numbers, see equations \eqref{eq:HNNdef}, \eqref{eq:HellNdef}, 
\item $H(k,n) = H_{k,1,1}(n)$ are Cohen's numbers, 
\item $H(n) = H(1,n) = H_{1,1,1}(n)$ are the Hurwitz class numbers,
\item $c_f^{\pm}(n)$ are the Fourier coefficients about $i\infty$ of the holomorphic ($+$) or non-holomorphic ($-$) part of a harmonic Maass form $f$, see equation \eqref{eq:hmffourier},
\item $K_{\kappa}(m,n;c)$ is the half integral weight Kloosterman sum, see equation \eqref{eq:Kloostermansumdef},
\item $\Ks_{N}(m,n;s)$ is the usual Kloosterman zeta function associated to the cusp $i\infty$ and with no multiplier, see equation \eqref{eq:KsNdef},
\item $\widetilde{\Ks}_{N}(m,n;s)$ is the modified Kloosterman zeta function associated to the cusp $0$ and with no multiplier, see equation \eqref{eq:KsNtildedef},
\item $\Ks_{k,4N}^{+}(m,n;s)$ is the plus space Kloosterman zeta function associated to the cusp $i\infty$ and with the theta multiplier, see equation \eqref{eq:Ksdef}.
\end{itemize}

\subsection*{Vector-valued notation, theta lifts, quadratic traces}
\begin{itemize}
\item $V$ is the three-dimensional space of traceless $2\times2$ matrices with rational entries,
\item $L \subseteq V$ is the lattice associated to $\Gamma_0(N)$, see equation \eqref{eq:latticedef},
\item $L'$ is the dual lattice of $L$, see equation \eqref{eq:latticedef},
\item $L^{-}$ is the lattice $L$ equipped with the quadratic form $-Q$ instead of $Q$,
\item $\mathrm{Mp}_2(\Z)$ is the metaplectic double cover of $\slz$, see equation \eqref{eq:Mp2Zdef},
\item $\rho_{L}$ is the Weil representation associated to $L$, $\rho_{L^{-}}$ is the dual Weil representation, see equation \eqref{eq:Weildef},
\item $\ef_{h}$ with $h \in L'\slash L$ is the standard basis of the group ring $\C[L'\slash L]$,
\item $\alpha_{\af}$, $\beta_{\af}$, $k_{\af}$, $m_{\af}$, $d_{\af}$, $\kappa_{\af}$ are certain parameters associated to a cusp $\af$, which we determine in Lemma \ref{lem:cuspparameters} and Corollary \ref{cor:cuspparameters} for $\af \in \{0,i\infty\}$,
\item $\Theta^{(\mathrm{M})}(\tau,z)$ is the Millson theta function, see equation \eqref{eq:ThetaMdef},
\item $\Theta^{(\mathrm{S})}(\tau,z)$ is the Shintani theta function, see equation \eqref{eq:ThetaSdef},
\item $I^{(S)}(f,\tau)$ is the regularized Shintani lift of $f$, see equation \eqref{eq:ISdef}. Its components are denoted by $I_h^{(S)}(f,\tau)$,
\item $I^{(M)}(f,\tau)$ is the regularized Millson lift of $f$, see equation \eqref{eq:IMdef}. Its components are denoted by $I_h^{(M)}(f,\tau)$,
\item $\Qc_{N,D}$ and $\Qc_{N,D,h}$ denote certain sets of integral binary quadratic forms of discriminant $D \in \Z \setminus \{0\}$, see equation \eqref{eq:Quadraticformssets},
\item $\tau_Q \in \H$ denotes the Heegner point associated to the integral binary quadratic form $Q \in \Qc_{N,D}$, see equation \eqref{eq:SQtauQdef},
\item $S_Q \subseteq \H$ denotes the geodesic associated to the integral binary quadratic form $Q \in \Qc_{N,D}$, see equation \eqref{eq:SQtauQdef},
\item $\tr(f;D,h)$ and $\tr(f;D)$ denote the imaginary quadratic ($D < 0$) resp.\ (regularized) real quadratic ($D > 0$) trace of the function $f$ with respect to the sets $\Qc_{N,D,h}$ and $\Qc_{N,D}$, see equations \eqref{eq:trimagdef} and \eqref{eq:trrealdef}.
\end{itemize}

\section{Preliminaries} \label{sec:prelim}

\subsection{Modular forms and harmonic Maass forms}

Let $\kappa \in \frac{1}{2}\Z$. We choose the principal branch of the complex square-root throughout. Let
\begin{align*}
\gamma &= \left(\begin{matrix} a & b \\ c& d \end{matrix}\right) \in \begin{cases}
\slz & \text{if } \kappa \in \Z, \\
\Gamma_0(4) & \kappa \in \frac{1}{2}+\Z,
\end{cases}
\qquad
\varepsilon_d \coloneqq \begin{cases}
1 & \text{if } d \equiv 1 \pmod*{4}, \\
i & \text{if } d \equiv 3 \pmod*{4},
\end{cases}
\end{align*}
for odd integers $d$ (this is guaranteed whenever $\gamma \in \Gamma_0(4)$) and $\left(\frac{c}{d}\right)$ denote the Kronecker symbol. Then, the \emph{Petersson slash operator} is defined as 
\begin{align} \label{eq:slashdef}
f\vert_{\kappa}\gamma(\tau) \coloneqq \begin{cases}
(c\tau+d)^{-\kappa} f(\gamma\tau) & \text{if } \kappa \in \Z, \\
\left(\frac{c}{d}\right)\varepsilon_d^{2\kappa}(c\tau+d)^{-\kappa} f(\gamma\tau) & \text{if } \kappa \in \frac{1}{2}+\Z.
\end{cases}
\end{align}
Furthermore, the \emph{weight $\kappa$ hyperbolic Laplace operator} is given by
\begin{align} \label{eq:Deltasplitting}
\Delta_{\kappa} \coloneqq - \xi_{2-\kappa} \xi_{\kappa} = -v^2\left(\frac{\partial^2}{\partial u^2}+\frac{\partial^2}{\partial v^2}\right) + i\kappa v\left(\frac{\partial}{\partial u} + i\frac{\partial}{\partial v}\right).
\end{align}
We define the modular objects that appear in this paper. 
\begin{defn} \label{defn:hmfdef}
Let $\mathcal{N} \in \N$. Suppose that $4 \mid \mathcal{N}$ whenever $\kappa \in \frac{1}{2} + \Z$. Let $f \colon \H \to \C$ be a smooth function.
\begin{enumerate}[label={\rm (\alph*)}]
\item We call $f$ a \emph{(holomorphic) modular form of weight $\kappa$ and level $\mathcal{N}$} if $f$ satisfies the following conditions:
\begin{enumerate}[label={\rm (\roman*)}]
\item We have $f\vert_{\kappa} \gamma = f$ for all $\gamma \in \Gamma_0(\mathcal{N})$,
\item the function $f$ is holomorphic on $\H$,
\item the function $f$ is holomorphic at every cusp of $\Gamma_0(\mathcal{N})$.
\end{enumerate}
We denote the vector space of modular forms by $M_{\kappa}(\mathcal{N})$.
\item If $f$ is a holomorphic modular form of weight $\kappa$ and level $\mathcal{N}$ that vanishes at all cusps of $\Gamma_0(\mathcal{N})$, then we call $f$ a \emph{cusp form of weight $\kappa$ and level $\mathcal{N}$}. We denote the vector space of cusp forms by $S_{\kappa}(N)$.
\item If $f$ is a holomorphic modular form of weight $\kappa$ and level $\mathcal{N}$ that is permitted to have poles at the cusps of $\Gamma_0(\mathcal{N})$, then we call $f$ a \emph{weakly holomorphic modular form of weight $\kappa$ and level $\mathcal{N}$}. We denote the vector space of cusp forms by $M_{\kappa}^{!}(N)$.
\item We call $f$ a \emph{harmonic Maass form of weight $\kappa$ and level $\mathcal{N}$} if $f$ satisfies the following conditions:
\begin{enumerate}[label={\rm  (\roman*)}]
\item We have $f\vert_{\kappa} \gamma = f$ for all $\gamma \in \Gamma_0(\mathcal{N})$,
\item we have $\Delta_{\kappa}f = 0$,
\item the function $f$ is of at most linear exponential growth towards all cusps of $\Gamma_0(\mathcal{N})$.
\end{enumerate}
We denote the vector space of harmonic Maass forms by $H_{\kappa}^{!}(\mathcal{N})$.
\end{enumerate}
Forms in \emph{Kohnen's plus space} are half integral weight forms, whose Fourier expansions are supported on indices $n$ satisfying $(-1)^{\kappa-\frac{1}{2}} n \equiv 0$, $1 \pmod*{4}$.
\end{defn}

The \emph{incomplete $\Gamma$-function} is given by
\begin{align} \label{eq:Gammadef}
\Gamma(s,z) \coloneqq \int_z^{\infty} t^{s-1}\mathrm{e}^{-t} \dm t, \qquad \re(s) > 0, \qquad z \in \C.
\end{align}
The incomplete $\Gamma$-function is multi-valued as a function of $s$. Each branch can be analytically continued in $s$ with removable singularities at $\Z_{\geq 0}$ whenever $z \neq 0$. We impose that the path of integration has to exclude $\R_{\leq 0}$, which yields the principal branch of the incomplete $\Gamma$-function. We refer the reader to \cite{brdieh}*{Section 2.2} for more details on this. As a function of the second variable, the incomplete $\Gamma$-function has the asymptotic behavior
\begin{align*}
\Gamma(s,v) \sim v^{s-1}\mathrm{e}^{-v}, \qquad \vert v \vert \to \infty, \qquad v \in \R,
\end{align*}
see \cite{table}*{item 8.3357} or \cite{nist}*{\S 8.11}. 

Since $\Delta_{\kappa}$ splits into $-\xi_{2-\kappa} \xi_{\kappa}$ by equation \eqref{eq:Deltasplitting}, the Fourier expansion of a weight $\kappa$ harmonic Maass form $f$ splits naturally into a \emph{holomorphic} and a \emph{non-holomorphic} part. If $\kappa \neq 1$, then these parts are of the shape
\begin{align} \label{eq:hmffourier}
f(\tau) = \sum_{n \gg -\infty} c_f^{+}(n) q^n + c_f^{-}(0)v^{1-\kappa} + \sum_{\substack{n \ll \infty \\ n \neq 0}} c_f^{-}(n)\Gamma(1-\kappa,-4\pi nv)q^n,
\end{align}
where the notation $\sum_{n \gg -\infty}$ abbreviates $\sum_{n \geq m_f}$ for some $m_f \in \Z$. The notation $\sum_{n \ll \infty}$ is defined analogously and similar expansions hold at the other cusps. The function
\begin{align*}
\sum_{n \leq 0} c_f^{+}(n) q^n
\end{align*}
is called the \emph{principal part of $f$}. Finally, we recall \emph{the (iterated) Maass raising operator} (see \cite{thebook}*{Section 5.1} for example)
\begin{align} \label{eq:Maassraising}
R_{\kappa} \coloneqq 2i\frac{\partial}{\partial\tau} + \frac{\kappa}{v}, \qquad R_{\kappa}^n \coloneqq R_{\kappa + 2(n-1)} \circ \ldots \circ R_{\kappa+2} \circ R_{\kappa}, \qquad n \in \N.
\end{align}
A quick computation shows that
\begin{align} \label{eq:Eisensteinraising}
R_{\kappa} v^{j} = (j+\kappa)v^{j-1}, \qquad R_{2-2k}^{k-1} v^{2k-1} = R_{-2} \circ \ldots \circ R_{2-2k} v^{2k-1} = (k-1)!v^k.
\end{align}

\subsection{Fourier expansions of Eisenstein series}
We collect the Fourier expansions of the various Eisenstein series appearing in this paper. To this end, we let $\re(s) > 1$ and define the usual \emph{level $N$ Kloosterman zeta function}
\begin{align} \label{eq:KsNdef}
\Ks_{N}(m,n;s) \coloneqq \sum_{\substack{c > 0 \\ N \mid c}} \frac{1}{c^s} \sum_{\substack{r \pmod*{c} \\ \gcd(c,r)=1}} e^{2\pi i\frac{mr^*+nr}{c}}, \qquad rr^* \equiv 1 \pmod*{c},
\end{align}
the \emph{modified level $N$ Kloosterman zeta function}
\begin{align} \label{eq:KsNtildedef}
\widetilde{\Ks}_{N}(m,n;s) \coloneqq \sum_{\substack{d > 0 \\ \gcd(d,N)=1}} \frac{1}{d^s} \sum_{\substack{r \pmod*{d} \\ \gcd(r,d)=1}} e^{2\pi i \frac{mr^*+nr}{d}},
\end{align}
and the \emph{plus space level $4N$ Kloosterman zeta function}
\begin{align} \label{eq:Ksdef}
\Ks_{\kappa,4N}^{+}(m,n;s) \sum_{c > 0} \frac{1+\left(\frac{4}{c}\right)}{(4Nc)^{s}} \sum_{\substack{r \pmod*{4Nc} \\ \gcd(4Nc,r)=1}} \left(\frac{4Nc}{r}\right)\varepsilon_r^{2\kappa} e^{2\pi i \frac{mr^*+nr}{4Nc}}.
\end{align}

We collect some evaluations first. To this end, we cite the following result.
\begin{lemma}[\cite{bemo2}*{Lemma 7.1}, \cite{mmrw}*{Lemma 7.1}] \label{lem:levellemma}
Let $N$ be square-free. Then we have
\begin{align*} 
\sum_{d \mid N} \mu(d) \chi_d(c)^2 = \begin{cases}
1 & \text{if } N \mid c, \\
0 & \text{if } N \nmid c.
\end{cases}
\end{align*}
\end{lemma}

Now, we are in position to evaluate the Kloosterman zeta functions in some cases.
\begin{lemma} \label{lem:Kloostermanzetaevals}
Let $N \in \N$ be square-free. Let $n \in \N$.
\begin{enumerate}[label={\rm  (\roman*)}]
\item Let $\sigma_s$ be the usual divisor sum (see equation \eqref{eq:divisordef}). We have
\begin{align*}
\Ks_{N}(0,n;s) = \frac{N^{1-s}}{\zeta(s)} \left(\prod_{\substack{p \text{ prime} \\ p\mid N}} \frac{1}{1-p^{-s}}\right) \sum_{\substack{d \mid N \\ \frac{N}{d} \mid n}} \mu(d) d^{-1} \sigma_{1-s}\left(\frac{nd}{N}\right).
\end{align*}
\item We have
\begin{align*} 
\widetilde{\Ks}_{N}(0,0;s) = \frac{\zeta(s-1)}{\zeta(s)} \prod_{\substack{p \text{ prime} \\ p\mid N}} \frac{1-p^{1-s}}{1-p^{-s}}.
\end{align*}
\item We have
\begin{align*}
\Ks_{N}(0,0;s) = \frac{\zeta(s-1)}{\zeta(s)}\sum_{\ell \mid N} \mu(\ell) \prod_{\substack{p \text{ prime} \\ p\mid \ell}} \frac{1-p^{1-s}}{1-p^{-s}}.
\end{align*}
\end{enumerate}
\end{lemma}

\begin{rmks}
\
\begin{enumerate}
\item We evaluate $\Ks_{\kappa,4N}^{+}(0,n;s)$ in Proposition \ref{prop:Kloostermanzetapeiwang} and Lemma \ref{lem:shimurasturmresult}.
\item Parts (ii) and (iii) appeared in \cite{bemo2}*{p.\ 32} in the case that $N=p$ is an odd prime.
\end{enumerate}
\end{rmks}

\begin{proof}
\
\begin{enumerate}[label={\rm  (\roman*)}]
\item This follows by \cite{moreno}*{p.\ $136$} (first displayed equation).
\item This follows by
\begin{align*}
\hspace*{\leftmargini} \sum_{\substack{d > 0 \\ \gcd(d,N)=1}} \frac{\varphi(d)}{d^{s}} = \prod_{\substack{p \text{ prime} \\ p\nmid N}} \sum_{j \geq 0} \frac{\varphi\left(p^j\right)}{p^{js}} = \prod_{\substack{p \text{ prime} \\ p\nmid N}} \Bigg(1+\sum_{j \geq 1} \frac{p^j-p^{j-1}}{p^{js}}\Bigg) = \prod_{\substack{p \text{ prime} \\ p\nmid N}} \frac{1-p^{-s}}{1-p^{1-s}}.
\end{align*}
\item This follows by Lemma \ref{lem:levellemma} and (ii). \qedhere
\end{enumerate}
\end{proof}

Now, we are in position to state the Fourier expansions.
\begin{lemma} \label{lem:Eisensteinfourier}
Let $k > 1$. Then we have
\begin{enumerate}[label={\rm  (\roman*)}]
\item The Fourier expansion of $\Fc_{2-2k,N}$ about $i\infty$ is given by
\begin{multline*}
\hspace*{\leftmargini} \Fc_{2-2k,N}(\tau) = v^{2k-1}
+ (-2i)^{2-2k} \pi  \sum_{n \geq 0} \Ks_{N}(0,n;2k) q^n \\
+ \frac{(-2i)^{2-2k}}{\Gamma(2k-1)} \sum_{n < 0} \Ks_{N}(0,n;2k) \Gamma\left(2k-1,4\pi\vt{n}v\right) q^n.
\end{multline*}

\item Let $W_N \coloneqq \left(\begin{smallmatrix} 0 & -1 \\ N & 1 \end{smallmatrix}\right)$ be the Fricke involution and extend the slash operator to $\mathrm{GL}_2(\R)^{+}$ as in \cite{bump}*{p.\ 41}. Then, the Fourier expansion of $\Fc_{2-2k,N}$ about $0$ is given by
\begin{multline*}
\hspace*{\leftmargini} \left(\Fc_{2-2k,N}\big\vert_{2-2k}W_N\right)(\tau) = \delta_{N=1} \frac{v^{2k-1}}{N^{k}} 
+ \frac{(-2i)^{2-2k} \pi}{N^k}  \sum_{n \geq 0} \widetilde{\Ks}_{N}(0,n;2k) q^n \\
+ \frac{(-2i)^{2-2k}}{N^k\Gamma(2k-1)} \sum_{n < 0} \widetilde{\Ks}_{N}(0,n;2k) \Gamma\left(2k-1,4\pi\vt{n}v\right) q^n.
\end{multline*}

\item The Fourier expansion of $\Fc_{\frac{3}{2}-k,4N}^{+}$ about $i\infty$ is given by
\begin{multline*}
\hspace*{\leftmargini} \Fc_{\frac{3}{2}-k,4N}^{+}(\tau) = \frac{2}{3}v^{k-\frac{1}{2}}
+ \frac{2}{3} \left(\frac{i}{2}\right)^{k-\frac{3}{2}} \pi \sum_{\substack{n \geq 0 \\ (-1)^{1-k}n \equiv 0,1 \pmod*{4}}} \Ks_{\frac{3}{2}-k,4N}^{+}\left(0,n;k+\frac{1}{2}\right) q^n \\
+ \frac{2\left(\frac{i}{2}\right)^{k-\frac{3}{2}} \pi}{3\Gamma\left(k-\frac{1}{2}\right)} \sum_{\substack{n < 0 \\ (-1)^{1-k}n \equiv 0,1 \pmod*{4}}} \Ks_{\frac{3}{2}-k,4N}^{+}\left(0,n;k+\frac{1}{2}\right) \Gamma\left(k-\frac{1}{2},4\pi\vt{n}v\right) q^n.
\end{multline*}

\item The Fourier expansion of $\Ec_{2k,N}$ about $i\infty$ is given by
\begin{align*}
\Ec_{2k,N}(\tau) = 1 + \frac{2\zeta(2k)}{\zeta(1-2k)}\sum_{n \geq 1} \Ks_{N}(0,n;2k) n^{2k-1} q^{n}.
\end{align*}
\end{enumerate}
\end{lemma}

\begin{proof}
\
\begin{enumerate}[label={\rm  (\roman*)}]
\item This is \cite{thebook}*{Theorem 6.15} (v).
\item We proceed as in \cite{bemo2}*{Lemma 2.3 (ii)} and use \cite{vass}*{Proposition 6.2.5} first. The usual coset representatives for $\Gamma_{\infty} \backslash \Gamma_0(N)$ yield
\begin{align*}
\left(\Fc_{2-2k,N}\big\vert_{2-2k}W_N\right)(\tau) 
&= \frac{v^{2k-1}}{2N^{k}} \sum_{\substack{(c,d) \in \Z^2 \setminus \left\{(0,0)\right\} \\ \gcd(Nc,d) = 1} }  \frac{1}{\left(c+d\tau\right) \left(c+d\overline{\tau}\right)^{2k-1}}.
\end{align*}
The claim follows by decomposing $c = dj+r$ with $j \in \Z$, $0 \leq r < d$ and using a generalized Lipschitz summation formula (see \cite{maass64}*{pp.\ 207}, \cite{siegel56}*{p.\ 366}). 
\item This can be found in \cite{bemo1}*{Proposition 2.5} (use \cite{bemo1}*{(2.4)} too). Since $k > 1$, the Kloosterman zeta functions converge and we may evaluate at spectral point $s = \frac{k}{2}+\frac{1}{4}$ directly. The Whittaker functions simplify by \cite{nist}*{items 13.18.2, 13.18.5}.
\item We combine (i) with \cite{thebook}*{Theorem 5.10 (iv) and Theorem 6.15 (ii)}, namely
\begin{align} \label{eq:Eisensteinshadow}
\xi_{2-2k} \Fc_{2-2k,N}(\tau) = (2k-1)\Ec_{2k,N}(\tau).
\end{align}
Note that $\Ks_{N}(0,n;2k) \in \R$ by Lemma \ref{lem:Kloostermanzetaevals} (i).
\qedhere
\end{enumerate}
\end{proof}

\subsection{Work of Pei and Wang} \label{subsec:peiwangprelim}

Let $N \in \N$, $N > 1$ be odd and square-free, $\ell \mid N$, and $L(s,\chi)$ be the usual Dirichlet $L$-function. Let
\begin{align} \label{eq:LNdef}
L_{N}(s,\chi) &\coloneqq L(s,\chi)\prod_{\substack{p \text{ prime} \\ p\mid N}} \left(1-\chi(p)p^{-s}\right) = \prod_{\substack{p \text{ prime} \\ p\nmid N}}\frac{1}{1-\chi(p)p^{-s}} = \sum_{\substack{n \geq 1 \\ \gcd(n,N)=1}} \frac{\chi(n)}{n^s}
\end{align}
be the \emph{incomplete Dirichlet $L$-function}. In addition, we abbreviate $\chi_{d} \coloneqq \left(\frac{d}{\cdot}\right)$ and
\begin{align} \label{eq:divisordef}
\sigma_{\ell,N,s}(r) \coloneqq \sum_{\substack{d \mid r \\ \gcd(d,\ell)=1 \\ \gcd\left(\frac{r}{d},\frac{N}{\ell}\right)=1}} d^s, \qquad \sigma_{N,s}(r) \coloneqq \sigma_{N,N,s}(r), \qquad \sigma_s(r) \coloneqq \sigma_{1,s}(r).
\end{align} 

\begin{rmk}
Let $N$ be square-free. Combining \cite{ramanujan}*{p.\ 343} with Lemma \ref{lem:levellemma} gives
\begin{align*}
\Ks_{N}(0,n;s) = \sum_{d \mid N} \mu(d) \frac{\sigma_{d,1-s}(n)}{L_d(s, \id)}.
\end{align*}
\end{rmk}

Following Pei and Wang, we define the \emph{generalized Hurwitz class numbers}
\begin{align} \label{eq:HNNdef}
H_{k,N,N}(n) 
&\coloneqq \begin{cases}
L_N\left(1-2k,\mathrm{id}\right) & \text{if } n=0, \\
L_N(1-k,\chi_t) \sum\limits_{\substack{a \mid m \\ \gcd(a,N) = 1}} \mu(a) \chi_t(a) a^{k-1} \sigma_{N,2k-1}\left(\frac{m}{a}\right) & \begin{array}{@{}l} \text{if } (-1)^kn=tm^2, \\ t \text{ fundamental}, \end{array}  \\
0 & \text{else},
\end{cases}
\end{align}
and
\small
\begin{multline} \label{eq:HellNdef}
H_{k,\ell,N}(n) \\
\coloneqq \begin{cases}
0 & \text{if } n=0, \\
L_{\ell}\left(1-k,\chi_{t}\right) \prod\limits_{\substack{p \text{ prime} \\ p\mid\frac{N}{\ell}}} \frac{1-\chi_t(p)p^{-k}}{1-p^{-2k}} \sum\limits_{\substack{a \mid m \\ \gcd(a,N) = 1}} \mu(a) \chi_t(a) a^{k-1} \sigma_{\ell,N,2k-1}\left(\frac{m}{a}\right) & \begin{array}{@{}l} \text{if } (-1)^kn=tm^2, \\ t \text{ fundamental}, \end{array} \\
0 & \text{else}
\end{cases}
\end{multline}
\normalsize
for $\ell \neq N$. Note that $H(k,n) = H_{k,1,1}(n)$ are Cohen's numbers and that $H_{1,1,1}(n) = H(n)$ are the classical Hurwitz class numbers.

\begin{rmk}
Following \cites{bemo1, bemo2}, we included the summation condition that $\gcd(a,N)=1$ in Pei's and Wang's sums running over $a \mid m$ inside the definition of $H_{k,\ell,N}(n)$ for all $\ell \mid N$.
\end{rmk}

Let $\Hs_{k,\ell,N}$ be as in equation \eqref{eq:Hsdef} and let $E_{k+\frac{1}{2}}^{+}(4N)$ be the Eisenstein subspace of forms in $M_{k+\frac{1}{2}}(4N)$ satisfying the plus space condition. 
\begin{lemma}[\cite{peiwang}*{Theorem 1}] \label{lem:peiwangmainresult}
\
\begin{enumerate}[label={\rm  (\roman*)}]
\item If $k=1$ and $N >1$ then 
$
\left\{\Hs_{1,\ell,N} \colon 1 < \ell \mid N \right\}
$
is a basis of the space $E_{\frac{3}{2}}^{+}(4N)$.
\item If $k > 1$ then
$
\left\{\Hs_{k,\ell,N} \colon \ell \mid N\right\}
$
is a basis of the space $E_{k+\frac{1}{2}}^{+}(4N)$.
\end{enumerate}
\end{lemma}

We emphasize that the function $\Hs_{1,1,N}$ is not modular, but the function $\Hs_{k,1,N}$ is modular whenever $k > 1$. Lastly, Pei and Wang worked out the Shimura lift of their generalized Cohen--Eisenstein series in \cite{peiwang}*{pp. 120--122}. Explicitly, we have
\begin{align*}
\mathscr{S}_D \left(\Hs_{k,\ell,N}\right) &= L_{\ell}\left(1-k, \chi_D\right) \Bigg(\prod_{\substack{p \text{ prime} \\ p\mid\frac{N}{\ell}}} \frac{1-\chi_D(p)p^{-k}}{1-p^{-2k}}\Bigg) G_{2k,\ell},
\end{align*}
where $G_{2k,\ell}$ can be evaluated in terms of $\Ec_{2k} = \Ec_{2k,1}$ (see equation \eqref{eq:Ecdef}) by
\begin{align*}
G_{2k,\ell}(\tau) = \frac{\zeta(1-2k)}{2}\sum_{r \mid \ell} \Big(\prod_{\substack{p \text{ prime} \\ p\mid r}} 1-p^{k-1}\Big) \sum_{t \mid \frac{N}{r}} \mu(t)\Ec_{2k}(rt\tau).
\end{align*}
Combining with Lemmas \ref{lem:Kloostermanzetaevals} (i) and \ref{lem:Eisensteinfourier} (iv), we find that
\begin{align*}
\Ec_{2k,p}(\tau) = \frac{2}{\zeta(1-2k)} \frac{1}{p^{2k-1}-1} \left(\frac{p^{2k}-p^{2k-1}}{p^{2k}-1}G_{2k,1}(\tau) - G_{2k,p}(\tau)\right)
\end{align*}
for an odd prime $p$.

\subsection{Vector-valued framework}

Let $N \in \N$, and define
\begin{align} 
V &\coloneqq \left\{\left(\begin{smallmatrix} x_2 & x_1 \\ x_3 & -x_2 \end{smallmatrix}\right) \colon x_1,x_2,x_3 \in \Q\right\}, \nonumber \\
L &\coloneqq \left\{\left(\begin{smallmatrix} b & \frac{c}{N} \\ a & -b \end{smallmatrix}\right) \colon a,b,c \in \Z\right\}, \label{eq:latticedef} \qquad L' \coloneqq \left\{\left(\begin{smallmatrix} b & \frac{c}{N} \\ a & -b \end{smallmatrix}\right) \colon a,c \in \Z, b \in \frac{1}{2N}\Z\right\}.
\end{align}
We equip $L$ with the quadratic form $Q(a,b,c) = Nb^2+ac$. The associated bilinear form is 
\begin{align*}
\left(\left(\begin{smallmatrix} b_1 & \frac{c_1}{N} \\ a_1 & -b_1 \end{smallmatrix}\right), \left(\begin{smallmatrix} b_2 & \frac{c_2}{N} \\ a_2 & -b_2 \end{smallmatrix}\right)\right) = 2Nb_1b_2+a_1c_2+a_2c_1,
\end{align*}
which yields signature $(2,1)$. The Grassmannian of lines in $V \otimes \R$ on which $Q$ is positive definite is isomorphic to $\H$. Note that $L'\slash L\cong \Z\slash2N\Z$, and we equip it with the quadratic form $x \mapsto x^2 / (4N)$. 

We recall the \emph{metaplectic double cover}
\begin{align} \label{eq:Mp2Zdef}
\mathrm{Mp}_2(\Z) \coloneqq \left\{ (\gamma, \phi) \colon \gamma = \left(\begin{smallmatrix} a & b \\ c & d \end{smallmatrix}\right)\in \slz, \ \phi\colon \H \rightarrow \C \text{ holomorphic}, \ \phi^2(\tau) = c\tau+d  \right\},
\end{align}
of $\slz$, which is generated by the pairs
\begin{align*}
\widetilde{T} \coloneqq \left(\left( \begin{smallmatrix} 1 & 1 \\ 0 & 1 \end{smallmatrix} \right),1\right), \qquad \widetilde{S} \coloneqq \left(\left( \begin{smallmatrix} 0 & -1 \\ 1 & 0 \end{smallmatrix}\right) ,\sqrt{\tau}\right),
\end{align*}
as well as the \emph{Weil representation $\rho_L$ associated to $L$}, which is defined on the generators by
\begin{align} \label{eq:Weildef}
\rho_L\left(\widetilde{T}\right)(\ef_h) &\coloneqq e^{2\pi iQ(h)} \ef_h, \qquad
\rho_L\left(\widetilde{S}\right)(\ef_h) \coloneqq \frac{e^{-\frac{\pi i}{4}}}{\sqrt{\vt{L'\slash L}}} \sum_{h' \in L'\slash L} e^{-2\pi i(h',h)_Q)} \ef_{h'},
\end{align}
where $\ef_{h}$ for $h \in L'\slash L$ is the standard basis of the group ring $\C[L'\slash L]$. We let $L^{-} \coloneqq (L,-Q)$ and call $\rho_{L^{-}}$ the \emph{dual Weil representation}. 

The \emph{vector-valued slash operator}
\begin{align} \label{eq:PeterssonVVdef}
f\big\vert_{\kappa,L}(\gamma,\phi) (\tau) \coloneqq \phi(\tau)^{-2\kappa}\rho_{L}^{-1}(\gamma,\phi)f(\gamma\tau)
\end{align}
encodes modularity with respect to $\rho_{L}$. Adapting Definition \ref{defn:hmfdef} accordingly yields the space of vector-valued holomorphic modular forms $M_{\kappa, L}$ and the space of vector-valued harmonic Maass forms $H_{\kappa, L}^{!}$.

Let $f$ be a smooth vector-valued function satisfying $f\big\vert_{\kappa,L}(\gamma,\phi) = f$ for every $(\gamma,\phi) \in \mathrm{Mp}_2(\Z)$. We denote the component functions of $f$ by $f_{h}$. If $N=1$ or $N=p$ is prime, the ``trace map''
\begin{align} \label{eq:plusspaceprojection}
\sum_{h \pmod*{2N}} f_h(\tau)\ef_h \mapsto \sum_{h \pmod*{2N}} f_h(4N\tau)
\end{align}
projects $f$ to a scalar-valued form of weight $\kappa$ and level $4N$ in Kohnen's plus space. We refer the reader to \cite{eiza}*{Theorem 5.6}, \cite{brufuim}*{Example 2.2}, \cite{alfesthesis}*{Subsection 7.3}, \cite{bor98}*{Example 2.4} for more details on this.

\subsection{Work of Alfes and Schwagenscheidt}

We let $\kappa \coloneqq k-1 \in \N$ in this subsection (recall that $k > 1$ throughout). Note that Alfes and Schwagenscheidt worked in signature $(1,2)$, while we work in signature $(2,1)$. In other words, we need to switch to the dual Weil representation in their discussion. 

\subsubsection{Theta kernels and theta lifts}
Let $\psi_{\kappa}$ and $\varphi_{\kappa}$ be the Millson and the Shintani Schwartz function as in \cite{alneschw18}*{p.\ 866}, \cite{alneschw21}*{p.\ 2316}. They give rise to the \emph{Millson} and \emph{Shintani theta functions}
\begin{align} 
\Theta^{(\mathrm{M})}(\tau,z) \coloneqq \sum_{h \in L'\slash L} \sum_{X \in L + h} \psi_{\kappa}(X,\tau,z) \ef_h, \label{eq:ThetaMdef} \\
\Theta^{(\mathrm{S})}(\tau,z) \coloneqq \sum_{h \in L'\slash L} \sum_{X \in L + h} \varphi_{\kappa}(X,\tau,z) \ef_h. \label{eq:ThetaSdef}
\end{align}
By \cite{alneschw18}*{Propsotion 3.1}, \cite{alneschw21}*{Proposition 4.1}, we have
\begin{enumerate}
\item the function $\tau \mapsto \Theta^{(\mathrm{M})}(\tau,z)$ has weight $\frac{1}{2}-\kappa$ for $\rho_{L^{-}}$,
\item the function  $z \mapsto \overline{\Theta^{(\mathrm{M})}(\tau,z)}$ has weight $-2\kappa$ for $\Gamma_0(N)$,
\item the function $\tau \mapsto \overline{\Theta^{(\mathrm{S})}(\tau,z)}$ has weight $\kappa+\frac{3}{2}$ for $\rho_{L}$, and
\item the function $z \mapsto \Theta^{(\mathrm{S})}(\tau,z)$ has weight $2\kappa+2$ for $\Gamma_0(N)$.
\end{enumerate}
Let $\sigma_{\af} \in \slz$ be a scaling matrix of the cusp $\af \in \Q \cup \{i\infty\}$ (that is we have $\sigma_{\af} \infty = \af$) and
\begin{align*}
\mathbb{F} \coloneqq \left\{\tau \in \H \colon -\frac{1}{2} \leq u < \frac{1}{2}, \ \vt{\tau} \geq 1 \right\}
\end{align*}
be the standard fundamental domain for $\slz$. Following Bruinier, Funke and Imamo\={g}lu \cite{brufuim}*{Subsection 5.4} and Alfes and Schwagenscheidt \cite{alneschw21}*{Subsection 2.2}, we truncate $\mathbb{F}$ at height $T > 0$ by defining
\begin{align*}
\mathbb{F}_T \coloneqq \left\{\tau \in \mathbb{F} \colon v \leq T\right\}.
\end{align*}
One recovers $\mathbb{F}$ from $\mathbb{F}_T$ by letting $T \to \infty$. Let $\mathcal{C}(N) \coloneqq \left(\Q \cup \{i\infty\}\right) \slash \Gamma_0(N)$ be a set of inequivalent cusps for $\Gamma_0(N)$ and $\alpha_{\af}$ be the width of the cusp $\af$. This established, a truncated fundamental domain for $\Gamma_0(N)$ is given by
\begin{align*}
\mathbb{F}(N)_T \coloneqq \bigcup_{\af \in \mathcal{C}(N)} \sigma_{\af} \mathbb{F}_T^{\alpha_{\af}}, \qquad \mathbb{F}_T^{\alpha_{\af}} \coloneqq \bigcup_{j=0}^{\alpha_{\af}-1} \left(\begin{smallmatrix} 1 & j \\ 0 & 1 \end{smallmatrix}\right)\mathbb{F}_T.
\end{align*} 

Let $z=x+iy$ and $\mu(z) \coloneqq \frac{\dm x \dm y}{y^2}$ be the usual hyperbolic measure. Let $F \in H_{-2\kappa}^{!}(N)$ and $G \in H_{2\kappa+2}^{!}(N)$ be scalar-valued harmonic Maass forms. Let $\mathrm{CT}_{s=0}(f(s))$ denote the constant term in the Laurent expansion of $f$ about $0$. Following \cite{brufuim}*{Definition 5.6}, Alfes and Schwagenscheidt define a \emph{regularized Millson theta lift} (see \cite{alneschw21}*{Definition 5.3}) 
\begin{multline} \label{eq:IMdef}
I^{(\mathrm{M})}\left(F, \tau\right) \\
\coloneqq \mathrm{CT}_{s=0} \sum_{\af \in \mathcal{C}(N)} \lim_{T \to \infty} \int_{\mathbb{F}(N)_T^{\alpha_{\af}}} \left(F\vert_{-2\kappa}\sigma_{\af}\right)(z) \overline{\left(\overline{\Theta^{(\mathrm{M})}}\vert_{-2\kappa}\sigma_{\af}\right)(\tau,z)} y^{-2\kappa-s} \dm\mu(z)
\end{multline}
as well as a \emph{regularized Shintani theta lift} (see \cite{alneschw21}*{Definition 5.1})
\begin{multline} \label{eq:ISdef}
I^{(\mathrm{S})}\left(G, \tau\right) \\
\coloneqq \mathrm{CT}_{s=0} \sum_{\af \in \mathcal{C}(N)} \lim_{T \to \infty} \int_{\mathbb{F}(N)_T^{\alpha_{\af}}} \left(G\vert_{2\kappa+2}\sigma_{\af}\right)(z) \overline{\left(\Theta^{(\mathrm{S})}\vert_{2\kappa+2}\sigma_{\af}\right)(\tau,z)} y^{2\kappa+2-s} \dm\mu(z).
\end{multline}
The transformation behavior of $\Theta^{(\mathrm{M})}$ resp.\ $\Theta^{(\mathrm{S})}$ implies that 
\begin{enumerate}
\item the function $I^{(\mathrm{M})}\left(F, \tau\right)$ transforms like a modular form of weight $\frac{1}{2}-\kappa$ for $\rho_{L^{-}}$ and
\item the function $I^{(\mathrm{S})}\left(G, \tau\right)$ transforms like a modular form of weight $\kappa+\frac{3}{2}$ for $\rho_{L}$.
\end{enumerate}
Moreover, we have
\begin{align*}
\Delta_{\frac{1}{2}-\kappa} I^{(\mathrm{M})}\left(F, \tau\right) &= 0, \qquad
\Delta_{\kappa+\frac{3}{2}} I^{(\mathrm{S})}\left(G, \tau\right) = 0,
\end{align*}
where the latter assertion requires that $\kappa > 0$, see \cite{alneschw21}*{Proposition 5.6}. The behavior towards the cusps follows by the Fourier expansions stated in \cite{alneschw18}*{Theorem 5.1} and \cite{alneschw21}*{Theorem 6.1}. Note that \cite{alneschw18}*{Theorem 5.1} assumes that $\xi_{-2\kappa}F$ is a cusp form. We extend this to the case $F = \Fc_{2-2k,p}$ in Subsection \ref{subsec:Millsonfourier}.

\subsubsection{Quadratic traces} \label{subsubsec:quadratictraces}
Let $D \in \Z \setminus \{0\}$. We define two sets of integral binary quadratic forms $Q = [a,b,c]$ of discriminant $D$ by
\begin{align} \label{eq:Quadraticformssets}
\begin{split}
\Qc_{N,D} &\coloneqq \left\{ax^2+bxy+cy^2 \colon a,b,c \in \Z, \ N \mid a, \ b^2-4ac=D \right\}, \\
\Qc_{N,D,h} &\coloneqq \left\{ax^2+bxy+cy^2 \colon a,b,c \in \Z, \ N \mid a, \ b^2-4ac=D, \ b \equiv h \pmod*{2N} \right\}.
\end{split}
\end{align}
They give rise to the usual ($0 > D \neq - Nn^2$ and $0 < D \neq n^2$), regularized ($0 < D = n^2$) and complementary ($0 > D = - Nn^2$) quadratic traces following \cite{brufuim}*{Section 3}, \cite{alneschw18}*{Section 5} and \cite{alneschw21}*{Sections 3 and 6}. Here, we recall that Bruinier, Funke and Imamo\={g}lu \cite{brufuim}*{Subsection 3.4} introduced a regularization of the cycle integral along infinite geodesics in weight $0$. Subsequently, Alfes and Schwagenscheidt extended their regularization to weights $2\kappa+2$, which we utilize in this paper as well. 

Let $f_1 \colon \H \to \C$ be of weight $0$ and $f_2 \in H_{2\kappa+2}^{!}(N)$. We recall
\begin{align} \label{eq:SQtauQdef}
S_{[a,b,c]} &\coloneqq \left\{z \in \H \colon a\vt{z}^2+bx+c=0\right\}, \qquad \tau_{[a,b,c]} \coloneqq \frac{-b}{2a} + \frac{\sqrt{\vt{D}}}{2\vt{a}}.
\end{align}
We define \emph{imaginary-} and \emph{real quadratic traces} by
\begin{alignat}{2}
\tr\left(f_1;D,h\right) &\coloneqq \sum_{Q \in \Qc_{N,4ND, h} \slash \Gamma_0(N)} \frac{f_1\left(\tau_Q\right)}{\vt{\left(\Gamma_0(N)\right)_{\tau_Q}}}, \qquad &&\text{ if } D < 0, \label{eq:trimagdef} \\
\tr\left(f_2;D, h\right) &\coloneqq \sum_{Q \in \Qc_{N,4ND, h} \slash \Gamma_0(N)} \int_{\left(\Gamma_0(N)\right)_Q\backslash S_Q}^{\reg} f_2(z)Q(z,1)^k \dm z, \qquad &&\text{ if } D > 0, \label{eq:trrealdef}
\end{alignat}
where we orient and regularize the cycle integral as in \cite{alneschw21}. We have
$
\bigcup_{h \in L' \slash L} \Qc_{N,D,h} = \Qc_{N,D},
$
and thus define
\begin{align*}
\tr\left(f_j;D\right) \coloneqq \sum_{h \in L' \slash L} \tr\left(f_j;D,h\right), \qquad j \in \{1,2\}.
\end{align*}

We omit an explicit description of the complementary trace $\tr^{c}\left(f_3,-Nn^2,h\right)$ of a function $f_3 \in H_{-2\kappa}^{!}(N)$, since we do not require it in the body of the paper.

\section{Proof of Theorem \ref{thm:CEpreimage}} \label{sec:proofCE}

\subsection{Evaluation of Kloosterman zeta functions}
We may reduce the general case of \emph{half integral weight Kloosterman sums}
\begin{align} \label{eq:Kloostermansumdef}
K_{\kappa}(m,n;c) \coloneqq \sum_{\substack{r \pmod*{c} \\ \gcd(c,r)=1}} \left(\frac{c}{r}\right)\varepsilon_r^{2\kappa} e^{2\pi i \frac{mr^*+nr}{c}}, \quad \kappa \in \frac{1}{2}+\Z, \quad 4 \mid c,
\end{align}
to the case of weight $\frac{1}{2}$ and $\frac{3}{2}$ Kloosterman sums, since
\begin{align} \label{eq:Kloostermanconnection}
\begin{split}
K_{\kappa+2}(m,n;c) &= K_{\kappa}(m,n;c) \\
&= (-1)^{\kappa-\frac{1}{2}}iK_{2-\kappa}(-n,-m;c) 
= (-1)^{\kappa-\frac{1}{2}}iK_{2-\kappa}(-m,-n;c),
\end{split}
\end{align}
compare with \cite{thebook}*{p.\ 268} or \cite{dit11annals}*{eq.\ (3.5)}. Alternatively, equation \eqref{eq:Kloostermanconnection} follows from \cite{andu}*{(1-9)} as well after noting that
\begin{align*}
e^{-2\pi i \frac{\kappa}{4}} \left(1+\left(\frac{4}{c}\right)\right)K_{\kappa}(m,n;c) \in \R.
\end{align*}

Equation $\eqref{eq:Kloostermanconnection}$ is the key ingredient to extending our results from \cite{bemo1} to higher weights. To be more precise, the proof of Theorem \ref{thm:CEpreimage} is based on the following proposition.
\begin{prop} \label{prop:Kloostermanzetapeiwang}
Let $N > 1$ be odd and square-free. Let $k > 1$ and $\ell \mid N$.
\begin{enumerate}[label={\rm (\roman*)}]
\item If $n = 0$ then
\begin{multline*}
\hspace*{\leftmargini} \frac{L_{4N}\left(2k-1,\id\right)}{L_{4N}\left(2k,\id\right)} \left(\frac{1+(-1)^ki}{2^{2k+1}-4} + \frac{1+(-1)^ki}{2^{2k+1}}\right) \prod_{\substack{p \text{ prime} \\ p\mid N}} \frac{p-1}{p\left(p^{2k-1}-1\right)} \\
= \begin{cases}
\overline{\Ks_{\frac{1}{2},4N}^{+}\left(0,0;k+\frac{1}{2}\right)} & \text{if } k \text{ is odd}, \\[10pt]
\overline{\Ks_{\frac{3}{2},4N}^{+}\left(0,0;k+\frac{1}{2}\right)} & \text{if } k \text{ is even}.
\end{cases}
\end{multline*}
\item If $n \geq 1$ with $(-1)^kn \equiv 0, 1 \pmod*{4}$ then
\begin{multline*}
\hspace*{\leftmargini} \sum_{\ell \mid N} \left(\prod_{\substack{p \text{ prime} \\ p\mid \frac{N}{\ell}}} \frac{p-1}{p\left(p^{2k-1}-1\right)}\right) \frac{\Gamma\left(k+\frac{1}{2}\right)}{(-2\pi i)^{k+\frac{1}{2}}} \frac{H_{k,\ell,N}(n)}{L_{\ell}(1-2k,\id)} \\
= \begin{cases}
\overline{\Ks_{\frac{1}{2},4N}^{+}\left(0,-n;k+\frac{1}{2}\right)} n^{k-\frac{1}{2}} & \text{if } k \text{ is odd}, \\[10pt]
\overline{\Ks_{\frac{3}{2},4N}^{+}\left(0,-n;k+\frac{1}{2}\right)} n^{k-\frac{1}{2}} & \text{if } k \text{ is even}.
\end{cases}
\end{multline*}
\end{enumerate}
\end{prop}

\begin{rmks}
\
\begin{enumerate}
\item If $N=1$, we have
\begin{align*}
\hspace*{\leftmargini} \Ks_{\frac{3}{2}-k,4}^{+}\left(0,0;k+\frac{1}{2}\right) = \frac{L_{4}\left(2k-1,\id\right)}{L_{4}\left(2k,\id\right)} \overline{\left(\frac{1+(-1)^ki}{2^{2k+1}-4} + \frac{1+(-1)^ki}{2^{2k+1}}\right)} = \frac{1-i}{2^{2k}} \frac{\zeta(2k-1)}{\zeta(2k)},
\end{align*}
which can be found in \cite{ibusa}*{Proposition 2.3} as well.
\item The evaluation of $\Ks_{\frac{3}{2}-k,4}^{+}\left(0,n;k+\frac{1}{2}\right)$ for $n \in \N$ can be found in \cite{andu}*{Proposition 5.7}.
\end{enumerate}
\end{rmks}

\subsection{Proof of Proposition \ref{prop:Kloostermanzetapeiwang}}
We follow \cite{bemo1}*{Sections 3 and 5}. We let $j\in \N$ and define
\begin{align*}
a_{\frac{1}{2}}\left(p^j,n\right) \coloneqq \begin{cases}
\sum\limits_{r=1}^{2^j} \left(\frac{2^j}{r}\right)\varepsilon_{r} e^{2\pi i\frac{nr}{2^j}} & \text{if } p=2, \\[10pt]
\varepsilon_{p^j}^{-1} \sum\limits_{r=1}^{p^j} \left(\frac{r}{p^j}\right)e^{2\pi i\frac{nr}{p^j}} & \text{if } p > 2,
\end{cases}
\quad 
a_{\frac{3}{2}}\left(p^j,n\right) \coloneqq \begin{cases}
\sum\limits_{r=1}^{2^j} \left(\frac{2^j}{r}\right)\varepsilon_{r}^{-1} e^{2\pi i\frac{nr}{2^j}} & \text{if } p=2, \\[10pt]
\varepsilon_{p^j} \sum\limits_{r=1}^{p^j} \left(\frac{r}{p^j}\right)e^{2\pi i\frac{nr}{p^j}} & \text{if } p > 2.
\end{cases}
\end{align*}
Note that
\begin{align*}
a_{\frac{1}{2}}(4,n) = \begin{cases}
1+i & \text{if } n \equiv 0,1 \pmod*{4}, \\
-1-i & \text{if } n \equiv 2,3 \pmod*{4},
\end{cases}
\qquad
a_{\frac{3}{2}}(4,n) = \begin{cases}
1-i & \text{if } n \equiv 0,3 \pmod*{4}, \\
-1+i & \text{if } n \equiv 1,2 \pmod*{4}.
\end{cases}
\end{align*}

We also define 
\begin{align*}
T_{4N,s}^{\chi}(n) &\coloneqq \sum_{\substack{0 < d \mid n \\ \gcd(d,4N)=1}} \mu(d)\chi(d)d^{s-1}\sigma_{4N,2s-1}\left(\frac{n}{d}\right).
\end{align*}

\begin{rmk}
We modify Pei and Wang's definition by imposing the extra summation condition $\gcd(d,4N)=1$, see the remark in \cite{bemo1}*{Subsection 2.4}.
\end{rmk}

We cite the following result.
\begin{lemma}[\protect{\cite{bemo1}*{Proposition 3.4}}] \label{lem:shimurasturmresult}
Let $\re(s) > 0$, $\kappa \in \{ \frac{1}{2}, \frac{3}{2}\}$, and write $n = tm^2$.
\begin{enumerate}[label={\rm (\roman*)}]
\item We have
\begin{multline*}
\hspace*{\leftmargini} \Ks_{\kappa,4N}^{+}\left(0,0;s+\frac{3}{2}\right) \\
= \frac{L_{4N}\left(2s+1,\id\right)}{L_{4N}\left(2s+2,\id\right)} \left(\sum_{j\geq2} \frac{a_{\kappa}\left(2^j,0\right) }{2^{j\left(s+\frac{3}{2}\right)}} + \frac{a_{\kappa}\left(4,0\right)}{2^{2\left(s+\frac{3}{2}\right)}}\right) \prod_{\substack{p \text{ prime} \\ p\mid N}} \sum_{j\geq1} \frac{a_{\kappa}\left(p^j,0\right)}{p^{j\left(s+\frac{3}{2}\right)}}.
\end{multline*}
\item If $n \neq 0$, we have
\begin{multline*}
\hspace*{\leftmargini} \Ks_{\kappa,4N}^{+}\left(0,n;s+\frac{3}{2}\right) \\
= \frac{L_{4N} (s+1, \chi_t)}{L_{4N} (2s+2, \id)} T_{4N,-s}^{\chi_t}(m) \left(  \sum_{j \geq 2} \frac{a_{\kappa}(2^j,n)}{2^{j(s + \frac{3}{2})}} + \frac{a_{\kappa}(4,n)}{2^{2(s + \frac{3}{2})}} \right) \prod_{\substack{p \text{ prime} \\ p\mid N}} \sum_{j \geq 1} \frac{a_{\kappa}(p^j,n)}{p^{j(s + \frac{3}{2})}}.
\end{multline*}
\end{enumerate}
\end{lemma}

In the course of proving their first theorem, Pei and Wang \cite{peiwang} showed the following identity.
\begin{lemma}[\protect{\cite{peiwang}*{Theorem 1 (II)}}] \label{lem:peiwangresult}
Let $N > 1$ be odd and square-free. Let $\ell \mid N$, and $n \geq 1$ satisfy $(-1)^k n = tm^2 \equiv 0, 1 \pmod*{4}$. Then we have
\begin{multline*}
\frac{1}{L_{\ell}(1-2k,\id)} H_{k,\ell,N}(n) = \frac{(-2\pi i)^{k+\frac{1}{2}}}{\Gamma\left(k+\frac{1}{2}\right)} \frac{L_{4 N} (k, \chi_{t})}{L_{4N} (2k, \id)} T_{4N,1-k}^{\chi_{t}}(m) \left(A_k(2,n) + \frac{1+(-1)^k i}{2^{2k+1}} \right) \\
\times \Bigg(\prod_{\substack{p \text{ prime} \\ p\mid \ell}} \left(A_k(p,n) - \frac{p-1}{p\left(p^{2k-1}-1\right)}\right) \Bigg) n^{k-\frac{1}{2}}.
\end{multline*}
\end{lemma}

\begin{proof}
Combine \cite{peiwang}*{(10a), (10b)} with \cite{bemo1}*{Lemma 3.3} and the penultimate sentence in Pei and Wang's aforementioned proof, see \cite{peiwang}*{p. 115}.
\end{proof}

Let $r \in \N$ and $p > 2$ be a prime. Pei and Wang \cite{peiwang}*{p.\ $106$} defined the numbers
\begin{multline*}
A_r(2,n) \coloneqq 2^{-(2r+1)}\left(1+(-1)^{r}i\right) \\
\times \begin{cases}
\frac{1-2^{(1-2r)\frac{\nu_2(n)-1}{2}}}{1-2^{1-2r}} - 2^{(1-2r)\frac{\nu_2(n)-1}{2}} & \text{if } 2 \nmid \nu_2(n), \\
\frac{1-2^{(1-2r)\frac{\nu_2(n)}{2}}}{1-2^{1-2r}} - 2^{(1-2r)\frac{\nu_2(n)}{2}} & \text{if } 2 \mid \nu_2(n), \frac{(-1)^{r}n}{2^{\nu_2(n)}} \equiv -1 \pmod*{4}, \\
\frac{1-2^{(1-2r)\frac{\nu_2(n)}{2}}}{1-2^{1-2r}} + 2^{(1-2r)\frac{\nu_2(n)}{2}}\left(1+2^{1-r}\left(\frac{\frac{(-1)^{r}n}{2^{\nu_2(n)}}}{2}\right)\right) & \text{if } 2 \mid \nu_2(n), \frac{(-1)^{r}n}{2^{\nu_2(n)}} \equiv 1 \pmod*{4},
\end{cases}
\end{multline*}
as well as
\begin{align*}
A_r(p,n) \coloneqq \begin{cases}
\frac{(p-1)\left(1-p^{(1-2r)\frac{\nu_p(n)-1}{2}}\right)}{p\left(p^{2r-1}-1\right)} - p^{(1-2r)\frac{\nu_p(n)+1}{2}-1} & \text{if } 2 \nmid \nu_p(n), \\
\frac{(p-1)\left(1-p^{(1-2r)\frac{\nu_p(n)}{2}}\right)}{p\left(p^{2r-1}-1\right)} + \left(\frac{\frac{(-1)^{r}n}{p^{\nu_p(n)}}}{p}\right) p^{(1-2r)\frac{\nu_p(n)+1}{2}-\frac{1}{2}} & \text{if } 2 \mid \nu_p(n).
\end{cases}
\end{align*}

\begin{rmk}
There is a small typo in the definition of $A_{r}(p,n)$ in \cite{peiwang}. Namely, if $2 \mid \nu_p(n)$, the sign of the second term in their definition of $A_r(p,n)$ should be flipped. We adjusted our definition of their numbers accordingly.
\end{rmk}

We begin by evaluating the local factors in Lemma \ref{lem:shimurasturmresult}. 
\begin{lemma} \label{lem:localconnection}
Let $k > 1$.
\begin{enumerate}[label={\rm (\roman*)}]
\item If $k$ is odd, then
\begin{align*}
\sum_{j\geq2} \frac{a_{\frac{1}{2}}\left(2^{j},n\right) }{2^{j\left(k+\frac{1}{2}\right)}} &= \overline{A_k(2,-n)}, \qquad \sum_{j\geq1} \frac{a_{\frac{1}{2}}\left(p^{j},n\right) }{p^{j\left(k+\frac{1}{2}\right)}} = \overline{A_k(p,-n)} = A_k(p,-n).
\end{align*}
\item If $k$ is even, then
\begin{align*}
\sum_{j\geq2} \frac{a_{\frac{3}{2}}\left(2^{j},n\right) }{2^{j\left(k+\frac{1}{2}\right)}} &= \overline{A_k(2,-n)}, \qquad \sum_{j\geq1} \frac{a_{\frac{3}{2}}\left(p^{j},n\right) }{p^{j\left(k+\frac{1}{2}\right)}} = \overline{A_k(p,-n)} = A_k(p,-n).
\end{align*}
\end{enumerate}
\end{lemma}

\begin{rmk}
Compare with \cite{bemo1}*{Lemma 5.2 (i)}, which evaluates the cases of weights $\frac{1}{2}$ and $\frac{3}{2}$.
\end{rmk}

\begin{proof}[Proof of Lemma \ref{lem:localconnection}]
We prove each case separately.
\begin{enumerate}[label={\rm (\roman*)}]
\item The case of weight $\frac{1}{2}$ is calculated in \cite{bemo1}, from which the case of odd $k$ follows by equation \eqref{eq:Kloostermanconnection}.

\item Let $k$ be even. By equation \eqref{eq:Kloostermanconnection}, we have
$
K_{\frac{3}{2}}(0,n;c) = (-1)^{\frac{3}{2}-\frac{1}{2}}iK_{\frac{1}{2}}(0,-n;c).
$
Hence we may use \cite{bemo1}*{(3.5), (3.6)} with $n \mapsto -n,$ and $1+i \mapsto  1-i$ to obtain
\begin{multline*}
\hspace*{\leftmargini} \sum_{j\geq2} \frac{a_{\frac{3}{2}}\left(2^{j},n\right) }{2^{j\left(k+\frac{1}{2}\right)}}
= \frac{1-i}{2^{k+1}} \\
\times \begin{cases} \frac{2^{-\nu_2(n)k}\left(2^{(\nu_2(n)+1)k}-2^{\frac{\nu_2(n)+1}{2}}\left(2^{2k}-1\right)\right)}{2^{2k}-2} & \text{if } 2 \nmid \nu_2(n), \\
\frac{2^{-(\nu_2(n)+1)k}\left(2^{(\nu_2(n)+2)k} - 2^{\frac{\nu_2(n)}{2}+2k+1} + 2^{\frac{\nu_2(n)}{2}+1}\right)}{2^{2k}-2} & \text{if } 2 \mid \nu_2(n), \frac{n}{2^{\nu_2(n)}} \equiv 1 \pmod*{4}, \\
\frac{2^{-(\nu_2(n)+2)k}\left(2^{\frac{\nu_2(n)}{2}+1}\left(2^k-2\right)\left(2^k+1\right)+2^{(\nu_2(n)+3)k}\right)}{2^{2k}-2} & \text{if } 2 \mid \nu_2(n), \frac{n}{2^{\nu_2(n)}} \equiv 7 \pmod*{8}, \\
\frac{2^{-(\nu_2(n)+2)k}\left(-2^{\frac{\nu_2(n)}{2}+1}\left(2^k+2^{2k}-2\right)+2^{(\nu_2(n)+3)k}\right)}{2^{2k}-2} & \text{if } 2 \mid \nu_2(n), \frac{n}{2^{\nu_2(n)}} \equiv 3 \pmod*{8},
\end{cases}
\end{multline*}
and
\begin{multline*}
\hspace*{\leftmargini} \sum_{j\geq1} \frac{a_{\frac{3}{2}}\left(p^{j},n\right) }{p^{j\left(k+\frac{1}{2}\right)}} = \\
 \begin{cases}
\left(1-p^{-2k}\right)\frac{p^{k(1-\nu_p(n))+\frac{\nu_p(n)+1}{2}}-p^{2k}}{p-p^{2k}} -1,
& \text{if } 2 \nmid \nu_p(n), \\
\frac{-p^{\frac{\nu_p(n)}{2}-\nu_p(n)k+2k}+p^{\frac{\nu_p(n)}{2}-\nu_p(n)k}\left(p^{2k} - p^{1-k} + p^k - p + 1\right) + p^{2k} - 1}{p^{2k} - p} -1,
& \text{if } 2 \mid \nu_p(n), \left(\frac{\frac{-n}{p^{\nu_p(n)}}}{p}\right) = 1, \\
\frac{-p^{\frac{\nu_p(n)}{2}-\nu_p(n)k + 2k} + p^{\frac{\nu_p(n)}{2}-\nu_p(n)k} \left(p^{2k} + p^{1-k} - p^k - p + 1\right)+ p^{2k} - 1}{p^{2k} - p} -1,
& \text{if } 2 \mid \nu_p(n), \left(\frac{\frac{-n}{p^{\nu_p(n)}}}{p}\right) = -1
\end{cases}
\end{multline*}
for $p > 2$. Note that $\nu_p(-n) = \nu_p(n)$. Since $k$ is even, we have $1+(-1)^k i = 1+i$ inside $A_k(2,n)$, while the evaluation of our local factor for $p=2$ yields $1-i$. Then, one can check by a direct computation that the expressions for each case do coincide if $p=2$. \qedhere
\end{enumerate}
\end{proof}

Now, we are in position to prove Proposition \ref{prop:Kloostermanzetapeiwang}
\begin{proof}[Proof of Proposition \ref{prop:Kloostermanzetapeiwang}]
We show the claim for even $k$, since the case of odd $k$ follows similarly.
\begin{enumerate}[label={\rm (\roman*)}]
\item Letting $\nu_2(n) \to \infty$ and $\nu_p(n) \to \infty$ in each arithmetic progression in Lemma \ref{lem:localconnection} yields
\begin{align*}
A_k(2,0) = \frac{1+(-1)^k i}{2^{2k+1}} \frac{1}{1-2^{1-2k}} = \frac{1+(-1)^k i}{2^{2k+1}-4}, \qquad A_k(p,0) = \frac{p-1}{p\left(p^{2k-1}-1\right)}.
\end{align*}
Lemma \ref{lem:shimurasturmresult} (i) implies the claim.
\item We let $s \to k-1$ in Lemma \ref{lem:shimurasturmresult} (ii), and use Lemma \ref{lem:localconnection} for even $k$ getting
\begin{multline*}
\hspace*{\leftmargini} \Ks_{\frac{3}{2},4N}^{+}\left(0,n;k+\frac{1}{2}\right) \\
= \frac{L_{4N} (k, \chi_t)}{L_{4N} (2k, \id)} T_{4N,1-k}^{\chi_t}(m) \left(\overline{A_k(2,-n)} + \frac{\overline{1-(-1)^k i}}{2^{2(k+\frac{1}{2})}} \right) \prod_{\substack{p \text{ prime} \\ p\mid N}} A_k(p,-n).
\end{multline*}
Following the proof of \cite{bemo1}*{Lemma 5.2 (ii)}, we write
\begin{align*}
\prod_{\substack{p \text{ prime} \\ p\mid N}} A_k(p,-n) &= \prod_{\substack{p \text{ prime} \\ p\mid N}} \left(A_k(p,-n) - \frac{p-1}{p\left(p^{2k-1}-1\right)} + \frac{p-1}{p\left(p^{2k-1}-1\right)} \right) \\
&= \sum_{\ell \mid N} \prod_{\substack{p \text{ prime} \\ p\mid \ell}} \left(A_k(p,-n) - \frac{p-1}{p\left(p^{2k-1}-1\right)}\right) \prod_{\substack{p \text{ prime} \\ p\mid \frac{N}{\ell}}} \frac{p-1}{p\left(p^{2k-1}-1\right)},
\end{align*}
and the result follows by Lemma \ref{lem:peiwangresult}. \qedhere
\end{enumerate}
\end{proof}

\subsection{Proof of Theorem \ref{thm:CEpreimage}}
Proposition \ref{prop:Kloostermanzetapeiwang} enables us to prove Theorem \ref{thm:CEpreimage}.
\begin{proof}[Proof of Theorem \ref{thm:CEpreimage}]
We first show the claim for even $k$. By equation \eqref{eq:Kloostermanconnection}, we have
\begin{align*}
\Ks_{\frac{3}{2}}(0,n;s) &= \Ks_{\frac{3}{2}-2m}(0,n;s), \qquad m \in 2\N,
\end{align*}
because the Kloosterman sum captures the weight entirely. Since $k > 1$ and $s = k-1 > 0$, the (twisted) Dirichlet series in Lemma \ref{lem:shimurasturmresult} have no poles. Combining Lemma \ref{lem:Eisensteinfourier} (iii) and Proposition \ref{prop:Kloostermanzetapeiwang} gives
\begin{multline*}
\xi_{\frac{3}{2}-k} \Fc_{\frac{3}{2}-k,4N}^{+}(\tau) \\
= \frac{2}{3}\left(k-\frac{1}{2}\right) \Bigg(1 + \sum_{\ell \mid N} \frac{1}{L_{\ell}(1-2k,\id)} \prod_{\substack{p \text{ prime} \\ p\mid \frac{N}{\ell}}} \frac{p-1}{p\left(p^{2k-1}-1\right)} \sum_{\substack{n > 0 \\ (-1)^kn \equiv 0,1 \pmod*{4}}} H_{k,\ell,N}(n) q^n \Bigg).
\end{multline*}
By equations \eqref{eq:HNNdef} and \eqref{eq:HellNdef}, we obtain
\begin{multline*}
\xi_{\frac{3}{2}-k} \Fc_{\frac{3}{2}-k,4N}^{+}(\tau) \\
= \frac{2}{3}\left(k-\frac{1}{2}\right) \sum_{\ell \mid N} \frac{1}{L_{\ell}(1-2k,\id)} \prod_{\substack{p \text{ prime} \\ p\mid\frac{N}{\ell}}} \frac{p-1}{p\left(p^{2k-1}-1\right)} \sum_{\substack{n \geq 0 \\ (-1)^kn \equiv 0,1 \pmod*{4}}} H_{k,\ell,N}(n) q^n,
\end{multline*}
and the claim follows by equations \eqref{eq:Hsdef} and \eqref{eq:LNdef}.

If $k$ is odd, we replace $\Ks_{\frac{3}{2},4N}^{+}\left(0,0;k+\frac{1}{2}\right)$ by $\Ks_{\frac{1}{2},4N}^{+}\left(0,0;k+\frac{1}{2}\right)$ and use the corresponding cases in Lemmas \ref{lem:shimurasturmresult} and \ref{lem:localconnection}. Note that Lemma \ref{lem:peiwangresult} does not depend on the parity of $k$, and thus the claim for odd $k$ follows analogously.
\end{proof}

\section{Fourier expansions of the regularized theta lifts of Eisenstein series} \label{sec:liftfourier}
The growth of the theta lifts towards the cusps can be deduced from their Fourier expansions. We work out the Fourier expansions for the inputs $\Ec_{2k,p}$ and $\Fc_{2-2k,p}$, respectively.

\subsection{Determining some parameters}
 We recall that the inequivalent cusps of $\Gamma_0(p)$ are represented by $0$ and $i\infty$, see \cite{cohstr}*{Example 6.3.25}.
\begin{lemma}[\cite{alfehl}*{Subsection 6.3}] \label{lem:cuspparameters}
Let $N=p$ be a prime.
\begin{enumerate}[label={\rm (\roman*)}]
\item The cusp $i\infty$ has width $\alpha_{\infty} = 1$ and parameter $\beta_{\infty} = \frac{1}{p}$.
\item The cusp $0$ has width $\alpha_{0} = p$ and parameter $\beta_{0} = 1$.
\end{enumerate}
\end{lemma}

We deduce the following assertions.
\begin{cor} \label{cor:cuspparameters}
Let $N=p$ be prime. Then, the holomorphic parts of the $h$-th components $I_h^{(\mathrm{M})}$ and $I_h^{(\mathrm{S})}$ of the regularized Millson and regularized Shintani lift may contain a constant term (with respect to the summation index of the Fourier expansion) if and only if $h = 0$. If $h=0$, then
\begin{align*}
k_0 = k_{\infty} = 0, \qquad m_0 = m_{\infty} = 0, \qquad d_0 = d_{\infty} = 1,
\end{align*}
in the notation of \cite{alneschw18}*{Theorem 5.1}. In the notation of \cite{alneschw21}*{Theorem 6.1}, we have
\begin{align*}
\kappa_0 = \kappa_{\infty} = 0.
\end{align*}
\end{cor}

\begin{rmk}
We utilized some of these assertions in \cite{bemo2}*{p.\ 26} as well.
\end{rmk}

\begin{proof}
Let $h \in L' \slash L \cong \Z \slash 2p\Z$. We recall that the constant term has to be invariant under $\rho_{L}(T)$, that is,
\begin{align*}
\rho_{L}(T)\mathfrak{e}_h = \mathfrak{e}_h.
\end{align*}
However, we also have
\begin{align*}
\rho_{L}(T)\mathfrak{e}_h = e^{2\pi i Q(h)}\mathfrak{e}_h = e^{2\pi i \frac{h^2}{4p}}\mathfrak{e}_h,
\end{align*}
which enforces $h = 0$ since $h \in \Z \slash 2p\Z$. Note that $h = 0 \in L' \slash L$ is equivalent to $h \in L$. According to\footnote{Explicitly stated in their preprint below Theorem 5.1 there.} \cite{alneschw18}*{p.\ 888}, we have $h \in L$ if and only if $k_{\af} = 0$. To see this, we recall that the cusps $i\infty$ resp.\ $0$ correspond to $\ell_{\infty} = \left(\begin{smallmatrix} 0 & 1 \\ 0 & 0\end{smallmatrix}\right)$ resp.\ $\ell_0 = \left(\begin{smallmatrix} 0 & 0 \\ -1 & 0\end{smallmatrix}\right)$. Let $\ell_{\af}$ be the isotropic line corresponding to the cusp $\af$. We recall that $k_{\af} \in \Q$ is defined by $\sigma_{\af}^{-1}h_{\af} = \left(\begin{smallmatrix} 0 & k_{\af} \\ 0 & 0\end{smallmatrix}\right)$ for some $h_{\af} \in \ell_{\af} \cap (L+h) \neq \emptyset$ and satisfies $0 \leq k_{\af} < \beta_{\af}$. Since $h = 0$, we obtain $h_{\af} \in \ell_{\af} \cap L$, so we may choose $h_{\af} = 0$. Hence, we must have $k_{\af} = 0$. Alternatively, we have $h = 0$ if and only if $\frac{k_{\af}}{\beta_{\af}} \in \Z$, which is stated on \cite{alneschw18}*{p.\ 888} as well. But $0 \leq \frac{k_{\af}}{\beta_{\af}} < 1$ by definition of $k_{\af}$, so $\frac{k_{\af}}{\beta_{\af}} \in \Z$ implies $\frac{k_{\af}}{\beta_{\af}} = 0$ and thus $k_{\af} = 0$ once more.
\end{proof}

\subsection{The regularized Shintani lift}
Following Alfes and Schwagenscheidt \cite{alneschw21}*{Theorem 6.1}, we state the Fourier expansion of the regularized Shintani lift of $\Ec_{2k,p}$.
\begin{prop} \label{prop:Shintanifourier}
Let $k \in 2\N$, $N=p$ be prime, and $I_h^{(\mathrm{S})}$ be the $h$-th component of the regularized Shintani theta lift. Then we have
\begin{align*}
I_h^{(\mathrm{S})}\left(\Ec_{2k,p}, \tau\right) = \delta_{h=0}(-1)^{\frac{k}{2}} \sqrt{p} \frac{\Gamma(k)\zeta(k)}{2^{k-1}\pi^k} - (-1)^{k-1} \sqrt{p} \sum_{n > 0} \tr\left(\Ec_{2k,p},n,h\right)q^n.
\end{align*}
\end{prop}

\begin{proof}
We use \cite{alneschw21}*{Theorem 6.1} and recall that they work with the dual Weil representation. This results in a change of the sign of the Fourier index inside the real quadratic trace as well as of orientation of the geodesics, for which reason we obtain an additional sign factor in front of the real quadratic trace. 

Since $\xi_{2k}\Ec_{2k,p} = 0$, the Fourier expansion of $I_h^{(\mathrm{S})}\left(\Ec_{2k,p}, \tau\right)$ contains neither imaginary quadratic nor complementary traces. By Lemma \ref{lem:Eisensteinfourier} (iv), the constant term of the holomorphic part of $\Ec_{2k,p}$ in its Fourier expansion about $i\infty$ equals $1$. Recall from \cite{bump}*{Lemma 2.1.1} that the (generalized) slash operator intertwines with the Maass lowering operator for any matrix of positive determinant. Hence, the slash operator and the $\xi_{\kappa}$-operator intertwine (see \cite{thebook}*{Theorem 5.10} too), and thus Lemma \ref{lem:Eisensteinfourier} (ii) implies that the holomorphic part of $\Ec_{2k,p}$ has no constant term in its Fourier expansion about $0$. Hence, the non-holomorphic part of the theta lift vanishes entirely and the principal part of the theta lift is constant. Using Lemma \ref{lem:cuspparameters} and Corollary \ref{cor:cuspparameters}, we obtain the constant term ($k \mapsto k-1$ in their result)
\begin{align*}
\frac{p^{k-\frac{1}{2}}}{2} \frac{1}{p^{k-1}}\left((-1)^k\zeta(1-k)+\zeta(1-k)\right) = \sqrt{p}\zeta(1-k),
\end{align*}
and the claim follows by the functional equation of the Riemann $\zeta$-function (see \cite{nist}*{\S 25.4}). 
\end{proof}

\subsection{The regularized Millson lift} \label{subsec:Millsonfourier}
We move to the Fourier expansion of the regularized Millson lift of $\Fc_{2-2k,p}$. To this end, we would like to use Alfes and Schwagenscheidt's earlier work \cite{alneschw18}*{Theorem 5.1}. However, their result does not apply directly here, because $\xi_{2-2k}\Fc_{2-2k,p}$ is not a cusp form (see equation \eqref{eq:Eisensteinshadow}). In other words, if the input function $f$ is a harmonic Maass form with a Fourier expansion of the shape as in equation \eqref{eq:hmffourier}, then we need to add the contributions arising from the Fourier coefficients $c_f^{-}(n)$ with $n \geq 0$ to the Fourier expansion of $I^{(\mathrm{M})}\left(f, \tau\right)$ stated there. In the case of $f = \Fc_{2-2k,p}$, we have $c_f^{-}(n) = 0$ for $n > 0$ by Lemma \ref{lem:Eisensteinfourier} (i), and hence it suffices to determine the contribution from $c_f^{-}(0)$. We do so in the following lemma.
\begin{lemma} \label{lem:liftconstantterms}
Let $k > 1$. Let $F \in H_{2-2k}^{!}(N)$ and $G \coloneqq \xi_{2-2k}F \in M_{2k}^{!}(N)$. Let $C_{(M)}^{\pm}(0)$ denote the constant terms of the holomorphic ($+$) resp.\ non-holomorphic part ($-$) of the regularized Millson lift of $F$, and let $C_{(S)}^{\pm}(0)$ denote the constant terms of the holomorphic ($+$) resp.\ non-holomorphic part ($-$) of the regularized Shintani lift of $G$ in its Fourier expansion about $i\infty$. Then we have
\begin{align*}
C_{(M)}^{\pm}(0) &= -\frac{1}{\sqrt{N}} \overline{C_{(S)}^{\mp}(0)}.
\end{align*}
\end{lemma}

\begin{proof}
The identity
\begin{align*}
C_{(M)}^{+}(0) = -\frac{1}{\sqrt{N}} \overline{C_{(S)}^{-}(0)}
\end{align*}
follows by comparing the Fourier expansions stated in \cite{alneschw18}*{Theorem 5.1} and \cite{alneschw21}*{Theorem 6.1}. Hence, it remains to prove the other identity. To this end, we use the differential equation \eqref{eq:thetaliftconnection} established in \cite{alneschw21}*{Proposition 5.5} applied to the full Fourier expansions of both theta lifts. It suffices to inspect the constant terms on both sides of the equation. 

We begin by justifying that $I^{(\mathrm{M})}\left(F, \tau\right)$ is a harmonic Maass form of weight $\frac{3}{2}-k$, since $k > 1$. To see this, \cite{alneschw21}*{Propositions 5.4 and 5.6} asserts modularity and harmonicity of $I^{(\mathrm{M})}\left(F, \tau\right)$. To verify the growth condition of $I^{(\mathrm{M})}\left(F, \tau\right)$, we note that the regularized Shintani lift maps $H_{2k}^{!}$ to $H_{k+\frac{1}{2}}^{!}$ according to \cite{alneschw21}*{Theorem 1.1}. Hence, the differential equation \eqref{eq:thetaliftconnection} implies the claimed growth condition since $\ker\left(\xi_{\frac{3}{2}-k}\right) = M_{\frac{3}{2}-k}^!$.

By equation \eqref{eq:hmffourier}, the function $F$ has the constant terms
\begin{align*}
c_F^{+}(0) + c_F^{-}(0)v^{2k-1}
\end{align*}
in its Fourier expansion about $i\infty$. Since $I^{(\mathrm{M})}\left(F, \tau\right)$ is a harmonic Maass form, equation \eqref{eq:hmffourier} implies that $I^{(\mathrm{M})}\left(F, \tau\right)$ has the constant terms
\begin{align*}
c_F^{+}(0)C_{(M)}^{+}(0) + c_F^{-}(0)C_{(M)}^{-}(0)v^{k-\frac{1}{2}}
\end{align*}
in its Fourier expansion about $i\infty$. We compute that
\begin{align*}
\xi_{\frac{3}{2}-k} \left(c_F^{+}(0)C_{(M)}^{+}(0) + c_F^{-}(0)C_{(M)}^{-}(0)v^{k-\frac{1}{2}}\right) &= \left(k-\frac{1}{2}\right) \overline{c_F^{-}(0)} \cdot \overline{C_{(M)}^{-}(0)}.
\end{align*}
On the other hand, $G$ has the constant term
\begin{align*}
\xi_{2-2k} \left(c_F^{+}(0) + c_F^{-}(0)v^{2k-1}\right) &= (2k-1)\overline{c_F^{-}(0)}
\end{align*}
in its Fourier expansion about $i\infty$. Since $G$ is a weakly holomorphic modular form and $k > 1$, $I^{(\mathrm{S})}\left(G, \tau\right)$ is a holomorphic modular form by \cite{alneschw21}*{Theorem 1.1}. Thus, $I^{(\mathrm{S})}\left(G, \tau\right)$ has the constant term
\begin{align*}
(2k-1) C_{(S)}^{+}(0) \overline{c_F^{-}(0)}
\end{align*}
in its Fourier expansion about $i\infty$. By virtue of the differential equation \eqref{eq:thetaliftconnection} relating both theta lifts, we obtain
\begin{align*}
\left(k-\frac{1}{2}\right) \overline{c_F^{-}(0)} \cdot \overline{C_{(M)}^{-}(0)} = -\frac{1}{2\sqrt{N}} (2k-1) C_{(S)}^{+}(0) \overline{c_F^{-}(0)},
\end{align*}
as desired.
\end{proof}

This enables us to state the Fourier expansion of $I_h^{(\mathrm{M})}\left(\Fc_{2-2k,p}, \tau\right)$.
\begin{prop} \label{prop:Millsonfourier}
Let $k \in 2\N$, $N=p$ be prime, and $I_h^{(\mathrm{M})}$ be the $h$-th component of the regularized Millson theta lift. Then we have
\begin{multline*}
I_h^{(\mathrm{M})}\left(\Fc_{2-2k,p}, \tau\right) \\
= \delta_{h=0} \frac{\Gamma(k)\zeta(k)}{2^{k-1}\pi^k}\left(\frac{\pi}{2^{k-1} \sqrt{p}} \left(p^k\Ks_{p}(0,0;2k) + p^{1-k}\widetilde{\Ks}_{p}(0,0;2k)\right) - (-1)^{\frac{k}{2}} v^{k-\frac{1}{2}} \right) \\
+ \frac{p^{\frac{k-1}{2}}}{2^{2k-1}\pi^{k-1}} \sum_{n > 0} \frac{1}{n^{\frac{k}{2}}} \tr\left(R_{2-2k}^{k-1} \Fc_{2-2k,p},-n,h\right)q^n \\
- \frac{1}{2^{2k}\pi^{k-\frac{1}{2}}} \sum_{n < 0} \frac{1}{\vt{n}^{k-\frac{1}{2}}} \overline{\tr\left(\xi_{2-2k}\Fc_{2-2k,p},-n,h\right)} \Gamma\left(k-\frac{1}{2},4\pi\vt{n}v\right)q^n.
\end{multline*}
\end{prop}

\begin{rmk}
The constant term of the holomorphic part can be simplified further by utilizing Lemma \ref{lem:Kloostermanzetaevals} (ii) and (iii), see \cite{bemo2}*{p.\ 32} as well.
\end{rmk}

\begin{proof}
We recall that Alfes and Schwagenscheidt work with the dual Weil representation and that we need to shift $k \mapsto k-1$ in their result. Since $\xi_{2-2k}\Fc_{2-2k,p} = (2k-1)\Ec_{2k,p}$ by equation \eqref{eq:Eisensteinshadow}, Lemma \ref{lem:liftconstantterms} implies that the Fourier expansion of $I_h^{(\mathrm{M})}\left(\Fc_{2-2k,p}, \tau\right)$ is given by \cite{alneschw18}*{Theorem 5.1} plus the term $-\frac{1}{\sqrt{p}} \overline{C_{(S)}^{+}(0)} v^{k-\frac{1}{2}}$.
\begin{enumerate}
\item According to Proposition \ref{prop:Shintanifourier}, we have
\begin{align*}
-\frac{1}{\sqrt{p}} \overline{C_{(S)}^{+}(0)} v^{k-\frac{1}{2}} = -\delta_{h=0}(-1)^{\frac{k}{2}} \frac{\Gamma(k)\zeta(k)}{2^{k-1}\pi^k} v^{k-\frac{1}{2}}.
\end{align*}

\item The constant term of the holomorphic part of $I^{(\mathrm{M})}\left(\Fc_{2-2k,p}, \tau\right)$ is obtained by evaluating the corresponding term in \cite{alneschw18}*{Theorem 5.1}. To this end, we combine Lemma \ref{lem:Eisensteinfourier} (i), (ii) as well as Lemmas \ref{lem:cuspparameters} and Corollary \ref{cor:cuspparameters}. 

\item The complementary trace inside the Fourier expansion of $I_h^{(\mathrm{M})}\left(\Fc_{2-2k,p}, \tau\right)$ vanishes, because the principal parts of $\Fc_{2-2k,p}$ about the cusps are constant, and the complementary trace depends on the non-constant terms in the principal part of $\Fc_{2-2k,p}$ about the inequivalent cusps (see equation \eqref{eq:hmffourier}). 

\item Both the holomorphic and non-holomorphic part of $I^{(\mathrm{M})}\left(\Fc_{2-2k,p}, \tau\right)$ can be read off from \cite{alneschw18}*{Theorem 5.1} directly noting that the two imaginary quadratic traces in the holomorphic part add up since $k$ is even. \qedhere
\end{enumerate}
\end{proof}

\section{Proof of Theorem \ref{thm:shitanilift} and Corollary \ref{cor:shintanicor}} \label{sec:proofShintani}

\subsection{Real quadratic traces of holomorphic Eisenstein series}

We evaluate the Fourier coefficients appearing in Proposition \ref{prop:Shintanifourier}. To prove Theorem \ref{thm:shitanilift}, it suffices to focus on the non-regularized quadratic traces. The following result is based on seminal work by Duke, Imamo\={g}lu and T\'{o}th \cite{dit11annals}*{Proposition 4} and Kohnen \cite{koh85}*{Proposition 5}.

\begin{prop} \label{prop:Eisensteintraces}
\
\begin{enumerate}[label={\rm (\roman*)}]
\item If $N \in \N$ and $D > 0$ is a non-square discriminant then
\begin{multline*}
\hspace*{\leftmargini} \sum_{Q \in \Qc_{N,D} \slash \Gamma_0(N)} \int_{\Gamma_0(N) \backslash S_Q} \Ec_{2k,N}(z) Q(z,1)^{k-1} \dm z \\
= D^{k-\frac{1}{2}} \frac{(-1)^k\sqrt{\pi}\Gamma(k)}{2^{2k-2}\Gamma\left(k+\frac{1}{2}\right)} \sum_{\substack{a \geq 1 \\ N \mid a}} \sum_{\substack{0 \leq b < 2a \\ b^2 \equiv D \pmod*{4a}}} \frac{1}{a^k}. 
\end{multline*}
\item Recall the Kloosterman sum $K_{\kappa}(m,n;c)$ from equation \eqref{eq:Kloostermansumdef}. Let $N$ be square-free and $D > 0$. We have
\begin{multline*}
\hspace*{\leftmargini} \sum_{\substack{a \geq 1 \\ N \mid a}} \sum_{\substack{0 \leq b < 2a \\ b^2 \equiv D \pmod*{4a}}} \frac{1}{a^k} = (-1)^{k} i \left(1-(-1)^{k}i\right) 4^{k-\frac{1}{2}} \zeta(k) \\
\times \sum_{\ell \mid N} \mu(\ell) \sum_{r \mid \ell} \mu(r)r^{-k} \sum_{\substack{c \geq 1 \\ \gcd(\ell,c) = 1}} \frac{1+\left(\frac{4}{c}\right)}{(4c)^{k+\frac{1}{2}}} K_{\frac{3}{2}-k}(0,-D;4c).
\end{multline*}
\item If $N = p$ is an odd prime then
\begin{multline*}
\hspace*{\leftmargini} \sum_{\substack{a \geq 1 \\ p \mid a}} \sum_{\substack{0 \leq b < 2a \\ b^2 \equiv D \pmod*{4a}}} \frac{1}{a^k} = (-1)^{k} i \left(1-(-1)^{k}i\right) 4^{k-\frac{1}{2}} \zeta(k) \\
\times \left(p^{-k} \Ks_{\frac{3}{2}-k,4}^{+}\left(0,-D;k+\frac{1}{2}\right) + \left(1-p^{-k}\right) \Ks_{\frac{3}{2}-k,4p}^{+}\left(0,-D;k+\frac{1}{2}\right) \right).
\end{multline*}
\end{enumerate}
\end{prop}

\begin{proof}[Proof of Proposition \ref{prop:Eisensteintraces}]
\
\begin{enumerate}[label={\rm (\roman*)}]
\item If $D \neq \square$ then the integral converges without regularization required. We follow an argument by Duke, Imamo\={g}lu and T\'{o}th \cite{dit11annals}*{Section 4} and adapt it slightly to higher levels. We give a proof for the convenience of the reader.

Since $\Ec_{2k,N}(z) Q(z,1)^{k-1} \dm z$ is a $\Gamma_0(N)$-invariant differential $1$-form, unfolding yields
\begin{align*}
\hspace*{\leftmargini} \sum_{Q \in \Qc_{N,D} \slash \Gamma_0(N)} \int_{\Gamma_0(N) \backslash S_Q} \Ec_{2k,N}(z) Q(z,1)^{k-1} \dm z 
= \sum_{Q \in \Qc_{N,D} \slash \Gamma_{\infty}} \int_{S_Q} Q(z,1)^{k-1} \dm z.
\end{align*}
Let $\Qc_{p, n}^{+}$ be the subset of integral binary quadratic forms with positive first coefficient $a$. Since the orientation of the integral is determined by $\mathrm{sgn}(a)$, mapping $a \mapsto -a$ gives
\begin{align*}
\hspace*{\leftmargini} \sum_{Q \in \Qc_{N,D} \slash \Gamma_0(N)} \int_{\Gamma_0(N) \backslash S_Q} \Ec_{2k,N}(z) Q(z,1)^{k-1} \dm z = 2 \sum_{Q \in \Qc_{N,D}^{+} \slash \Gamma_{\infty}} \int_{S_Q} Q(z,1)^{k-1} \dm z.
\end{align*}
By \cite{ilt22}*{Lemma 3.2}, the map
\begin{align*}
\left\{(a,b) \colon a > 0, N \mid a, b \pmod*{2a}, b^2 \equiv D \pmod*{4a}\right\} &\to \Gamma_{\infty} \backslash \Qc_{N,D}^{+} \\
(a,b) &\mapsto \left[a,b,\frac{b^2-D}{4a}\right]
\end{align*}
is a bijection. We infer
\begin{multline*}
\hspace*{\leftmargini} \sum_{Q \in \Qc_{N,D} \slash \Gamma_0(N)} \int_{\Gamma_0(N) \backslash S_Q} \Ec_{2k,N}(z) Q(z,1)^{k-1} \dm z \\
= 2 \sum_{\substack{a \geq 1 \\ N \mid a}} \sum_{\substack{0 \leq b < 2a \\ b^2 \equiv D \pmod*{4a}}} \int_{S_{\left[a,b,\frac{b^2-D}{4a}\right]}} 
\left(az^2+bz+\frac{b^2-D}{4a}\right)^{k-1} \dm z.
\end{multline*}
We recall that the semicircle $S_{Q}$, $Q=\big[a,b,\frac{b^2-D}{4a}\big]$, is centered at $-\frac{b}{2a}$ and of radius $\frac{\sqrt{D}}{2a}$, because its endpoints are the two real-quadratic roots $\frac{-b\pm\sqrt{D}}{2a} $ of $Q$ and its apex is $\tau_Q$. Hence, a parametrization of $S_Q$ is given by
\begin{align*}
\theta &\mapsto -\frac{b}{2a} + \frac{\sqrt{D}}{2a}e^{i\theta}, \qquad 0 < \theta < \pi.
\end{align*}
It follows that (\cite{dit11annals}*{p.\ 968})
\begin{align*}
\hspace*{\leftmargini} Q(z,1) = \frac{D}{4a} \left(e^{i\theta}-1\right)\left(e^{i\theta}+1\right), \quad
\dm z = i\frac{\sqrt{D}}{2a}e^{i\theta} \dm\theta, \quad \frac{ie^{i\theta}}{\left(e^{i\theta}-1\right)\left(e^{i\theta}+1\right)} = \frac{1}{2\sin(\theta)}.
\end{align*}
Thus, we obtain
\begin{multline*}
\hspace*{\leftmargini} \sum_{Q \in \Qc_{N,D} \slash \Gamma_0(N)} \int_{\Gamma_0(N) \backslash S_Q} \Ec_{2k,N}(z) Q(z,1)^{k-1} \dm z \\
= 2^{1-2k} D^{k-\frac{1}{2}} \int_{0}^{\pi} \left(e^{2i\theta}-1\right)^{k} \frac{1}{\sin(\theta)} \dm \theta \sum_{\substack{a \geq 1 \\ N \mid a}} \sum_{\substack{0 \leq b < 2a \\ b^2 \equiv D \pmod*{4a}}} \frac{1}{a^k}. 
\end{multline*}
Lastly, we compute that
\begin{align*}
\hspace*{\leftmargini} \int_{0}^{\pi} \left(e^{2i\theta}-1\right)^{k} \frac{1}{\sin(\theta)} \dm \theta = \left[-2 (-1)^k e^{i\theta} {}_2F_1\left(\frac{1}{2},1-k;\frac{3}{2};e^{2i\theta}\right)\right]_{\theta = 0}^{\theta=\pi} = \frac{2(-1)^k\sqrt{\pi}\Gamma(k)}{\Gamma\left(k+\frac{1}{2}\right)}.
\end{align*}
This proves (i).

\item Let
\begin{align*}
K_{k+\frac{1}{2}}^{+}(m,n;c) \coloneqq \left(1-(-1)^{k}i\right)\left(1+\left(\frac{4}{c}\right)\right) \frac{1}{4c} K_{k+\frac{1}{2}}(m,n;4c)
\end{align*}
be \emph{Kohnen's plus space Kloosterman sum} \cite{koh85}*{eq.\ (26)}. According to Kohnen \cite{koh85}*{Proposition 5} (in a slightly modified form stated by Duke, Imamo\={g}lu and T\'{o}th \cite{dit11annals}*{Proposition 3}), we have
\begin{align*}
\sum_{\substack{0 \leq b < 2a \\ b^2 \equiv D \pmod*{4a}}} 1 = \sum_{d \mid a} \sqrt{\frac{a}{d}} K_{k+\frac{1}{2}}^{+}\left(0,D;\frac{a}{d}\right) = \sum_{d \mid a} \sqrt{d} K_{k+\frac{1}{2}}^{+}\left(0,D;d\right).
\end{align*}
Since $N$ is square-free, we may use Lemma \eqref{lem:levellemma}. Thus, we obtain
\begin{align*}
\sum_{\substack{a \geq 1 \\ N \mid a}} \sum_{\substack{0 \leq b < 2a \\ b^2 \equiv D \pmod*{4a}}} \frac{1}{a^k} 
&= \sum_{\ell \mid N} \mu(\ell) \sum_{a \geq 1} \frac{\chi_{\ell}(a)^2}{a^k} \sum_{dr = a} \sqrt{d} K_{k+\frac{1}{2}}^{+}\left(0,D;d\right) \\
&= \sum_{\ell \mid N} \mu(\ell) \sum_{r \geq 1} \frac{\chi_{\ell}(r)^2}{r^k} \sum_{d \geq 1} \frac{\chi_{\ell}(d)^2}{d^{k-\frac{1}{2}}} K_{k+\frac{1}{2}}^{+}\left(0,D;d\right).
\end{align*}
We infer
\begin{align*}
\sum_{r \geq 1} \frac{\chi_{\ell}(r)^2}{r^k} = L_{\ell}(k,\id) = \zeta(k) \prod_{\substack{p \text{ prime} \\ p\mid \ell}} \left(1-p^{-k}\right) = \zeta(k) \sum_{r \mid \ell} \mu(r)r^{-k},
\end{align*}
and
\begin{multline*}
\hspace*{\leftmargini} \sum_{d \geq 1} \frac{\chi_{\ell}(d)^2}{d^{k-\frac{1}{2}}} K_{k+\frac{1}{2}}^{+}\left(0,D;d\right) 
= \left(1-(-1)^{k}i\right) 4^{k-\frac{1}{2}} \sum_{\substack{c \geq 1 \\ \gcd(\ell,c) = 1}} \frac{1+\left(\frac{4}{c}\right)}{(4c)^{k+\frac{1}{2}}} K_{k+\frac{1}{2}}(0,D;4c) \\
= (-1)^{k} i \left(1-(-1)^{k}i\right) 4^{k-\frac{1}{2}} \sum_{\substack{c \geq 1 \\ \gcd(\ell,c) = 1}} \frac{1+\left(\frac{4}{c}\right)}{(4c)^{k+\frac{1}{2}}} K_{\frac{3}{2}-k}(0,-D;4c)
\end{multline*}
by equation \eqref{eq:Kloostermanconnection}. Combining, we obtain (ii).

\item We abbreviate
\begin{align*}
h(c) \coloneqq \frac{1+\left(\frac{4}{c}\right)}{(4c)^{k+\frac{1}{2}}} K_{\frac{3}{2}-k}(0,-D;4c).
\end{align*}
On one hand, if $N = p$ is an odd prime then (ii) evaluates to
\begin{align*}
\hspace*{\leftmargini} \sum_{\ell \mid p} \mu(\ell) \sum_{r \mid \ell} \mu(r)r^{-k} \sum_{\substack{c \geq 1 \\ \gcd(\ell,c) = 1}} h(c) = \sum_{c \geq 1} h(c) - \sum_{\substack{c \geq 1 \\ \gcd(p,c) = 1}} h(c) + p^{-k}\sum_{\substack{c \geq 1 \\ \gcd(p,c) = 1}} h(c).
\end{align*}
On the other hand, Lemma \ref{lem:levellemma} implies that
\begin{align*}
\sum_{\substack{c \geq 1 \\ p \mid c}} h(c) = \sum_{c \geq 1} h(c) - \sum_{\substack{c \geq 1 \\ \gcd(p,c) = 1}} h(c).
\end{align*}
Hence,
\begin{multline*}
\hspace*{\leftmargini} \sum_{\ell \mid p} \mu(\ell) \sum_{r \mid \ell} \mu(r)r^{-k} \sum_{\substack{c \geq 1 \\ \gcd(\ell,c) = 1}} h(c) = \sum_{\substack{c \geq 1 \\ p \mid c}} h(c) + p^{-k} \sum_{\substack{c \geq 1 \\ \gcd(p,c) = 1}} h(c) \\
= \sum_{\substack{c \geq 1 \\ p \mid c}} h(c) + p^{-k}\Bigg(\sum_{c \geq 1} h(c) - \sum_{\substack{c \geq 1 \\ p \mid c}} h(c)\Bigg) 
= p^{-k}\sum_{c \geq 1} h(c) + \left(1-p^{-k}\right)\sum_{\substack{c \geq 1 \\ p \mid c}} h(c).
\end{multline*}
This proves (iii). \qedhere
\end{enumerate}
\end{proof}

We infer the following immediate corollary.
\begin{cor} \label{cor:Eisensteintraces}
If $N = p$ is an odd prime and $1 \leq D \neq \square$ then
\begin{multline*}
\sum_{Q \in \Qc_{p,D} \slash \Gamma_0(p)} \int_{\Gamma_0(p) \backslash S_Q} \Ec_{2k,p}(z) Q(z,1)^{k-1} \dm z = (-1)^{\lfloor \frac{k}{2} \rfloor} \frac{\Gamma(k)\zeta(k)}{2^{k-1} \pi^k} \\
\times \left(p^{-k} \frac{H_{k,1,1}(D)}{\zeta(1-2k)} + \left(1-p^{-k}\right) \left(\frac{p-1}{p^{2k}-p} \frac{H_{k,1,p}(D)}{\zeta(1-2k)} + \frac{H_{k,p,p}(D)}{L_{p}(1-2k,\id)} \right) \right).
\end{multline*}
\end{cor}

\begin{proof}
Firstly, Proposition \ref{prop:Eisensteintraces} yields
\begin{multline*}
\sum_{Q \in \Qc_{p,D} \slash \Gamma_0(p)} \int_{\Gamma_0(p) \backslash S_Q} \Ec_{2k,p}(z) Q(z,1)^{k-1} \dm z = 2i\left(1-(-1)^{k}i\right) \frac{D^{k-\frac{1}{2}} \sqrt{\pi}\Gamma(k)\zeta(k)}{\Gamma\left(k+\frac{1}{2}\right)} \\
\times \left(p^{-k} \Ks_{\frac{3}{2}-k,4}^{+}\left(0,-D;k+\frac{1}{2}\right) + \left(1-p^{-k}\right) \Ks_{\frac{3}{2}-k,4p}^{+}\left(0,-D;k+\frac{1}{2}\right) \right).
\end{multline*}
Secondly, we insert Proposition \ref{prop:Kloostermanzetapeiwang} (ii) getting
\begin{align*}
\Ks_{\frac{3}{2}-k,4}^{+}\left(0,-D;k+\frac{1}{2}\right) &= \frac{1}{D^{k-\frac{1}{2}}} \frac{\Gamma\left(k+\frac{1}{2}\right)}{(2\pi i)^{k+\frac{1}{2}}} \frac{H_{k,1,1}(D)}{\zeta(1-2k)}, \\
\Ks_{\frac{3}{2}-k,4p}^{+}\left(0,-D;k+\frac{1}{2}\right) &= \frac{1}{D^{k-\frac{1}{2}}} \frac{\Gamma\left(k+\frac{1}{2}\right)}{(2\pi i)^{k+\frac{1}{2}}} \left(\frac{p-1}{p^{2k}-p} \frac{H_{k,1,p}(D)}{\zeta(1-2k)} + \frac{H_{k,p,p}(D)}{L_{p}(1-2k,\id)} \right).
\end{align*}
Noting that $\frac{i\left(1-(-1)^{k}i\right)}{i^{k+\frac{1}{2}}} = (-1)^{\lfloor \frac{k}{2} \rfloor} \sqrt{2}$ and combining yields the claim.
\end{proof}

\subsection{Proof of Theorem \ref{thm:shitanilift} and Corollary \ref{cor:shintanicor}}

\begin{proof}[Proof of Theorem \ref{thm:shitanilift} and Corollary \ref{cor:shintanicor}]
We abbreviate
\begin{align*}
f_{1}(\tau) \coloneqq \sum_{h \pmod*{2p}} I_h^{(\mathrm{S})}\left(\Ec_{2k,p}, 4p\tau\right)
\end{align*}
and
\begin{multline*}
f_{2}(\tau) \coloneqq (-1)^{\frac{k}{2}} \sqrt{p}\frac{\Gamma(k)\zeta(k)}{2^{k-1} \pi^k} \\
\times
\left(p^{-k}\frac{\Hs_{k,1,1}(\tau)}{\zeta(1-2k)} + \left(1-p^{-k}\right)\left(\frac{p-1}{p^{2k}-p} \frac{\Hs_{k,1,p}(\tau)}{\zeta(1-2k)} + \frac{\Hs_{k,p,p}(\tau)}{L_{p}(1-2k,\id)}\right)\right).
\end{multline*}
We note that
\begin{align*}
c_{f_{1}}(0) = (-1)^{\frac{k}{2}} \sqrt{p} \frac{\Gamma(k)\zeta(k)}{2^{k-1} \pi^k} = c_{f_{2}}(0) \neq 0
\end{align*}
for even $k$ by Proposition \ref{prop:Shintanifourier}. Hence, we obtain $f_{1} \in E_{k+\frac{1}{2}}^{+}(4p)$, because $p$ is prime and thus the plus space condition is satisfied. Furthermore, we have $f_{2} \in E_{k+\frac{1}{2}}^{+}(4p)$ by \cite{peiwang}*{Theorem 1 (II)} (see Lemma \ref{lem:peiwangmainresult} (ii) too). By virtue of Proposition \ref{prop:Shintanifourier}, the positive indexed Fourier coefficients of $f_{1}$ are given by
\begin{multline*}
-(-1)^{k-1}\sqrt{p} \sum_{h \pmod*{2p}} \sum_{n \geq 1} \tr\left(\Ec_{2k,p},n,h\right)q^{4pn} 
= (-1)^{k}\sqrt{p} \sum_{n \geq 1} \sum_{h \pmod*{2p}} \tr\left(\Ec_{2k,p},\frac{n}{4p},h\right)q^n \\
= (-1)^{k}\sqrt{p} \sum_{n \geq 1} \left[\sum_{Q \in \Qc_{N,n} \slash \Gamma_0(p)} \int_{\Gamma_0(p) \backslash S_Q}^{\reg} \Ec_{2k,p}(z) Q(z,1)^{k-1} \dm z \right] q^n.
\end{multline*}
According to Corollary \ref{cor:Eisensteintraces}, we deduce that the Fourier coefficients of $f_{1}$ and $f_{2}$ coincide on the infinite sequence of all non-square indices $n \geq 1$. Thus, we infer
\begin{align*}
f_{1} - f_{2} \eqqcolon f_3 \in S_{k+\frac{1}{2}}(4p),
\end{align*}
where the Fourier coefficients of $f_3$ are supported on square indices (in particular, the plus space condition is satisfied). However, such a form is a linear combination of weight $\frac{1}{2}$ or $\frac{3}{2}$ theta series according to a result of Vign\'eras \cite{vign77}*{Theorem 3} (see Bruinier \cite{bru98} for a different proof as well). Since $k > 1$, this implies that this linear combination has to be trivial and we infer $f_3 = 0$, which proves Theorem \ref{thm:shitanilift}. Corollary \ref{cor:shintanicor} then follows by comparing the square-indexed Fourier coefficients in Theorem \ref{thm:shitanilift}.
\end{proof}

\section{Proof of Theorem \ref{thm:millsonlift} and Corollary \ref{cor:millsoncor}} \label{sec:proofMillson}

We move to the proof of Theorem \ref{thm:millsonlift}.
\begin{proof}[Proof of Theorem \ref{thm:millsonlift}]
We abbreviate
\begin{align*}
g_{1}(\tau) \coloneqq \sum_{h \pmod*{2p}} I_h^{(\mathrm{M})}\left(\Fc_{2-2k,p}, 4p\tau\right),
\end{align*}
and
\begin{align*}
g_{2}(\tau) \coloneqq 3 (-1)^{\frac{k}{2}-1} \frac{\Gamma(k)\zeta(k)}{2^{k} \pi^k} (4p)^{k-\frac{1}{2}} \left(p^{-k} \Fc_{\frac{3}{2}-k,4}^{+}(\tau)  + \left(1-p^{-k}\right) \Fc_{\frac{3}{2}-k,4p}^{+}(\tau)\right).
\end{align*}
We compute $\xi_{\frac{3}{2}-k}g_1$ and $\xi_{\frac{3}{2}-k}g_2$. According to Proposition \ref{prop:Millsonfourier}, the non-holomorphic part of $g_1$ is given by
\begin{multline*}
-\sum_{h \pmod*{2p}} \sum_{m < 0} \frac{1}{2(4\pi\vt{m})^{k-\frac{1}{2}}} \overline{\tr\left(\xi_{2-2k}\Fc_{2-2k,p};m;h\right)} \Gamma\left(k-\frac{1}{2}, 4\pi \vt{m} (4pv)\right) q^{4pm} \\
= -\sum_{m < 0} \frac{2k-1}{2(4\pi\frac{\vt{m}}{4p})^{k-\frac{1}{2}}} \overline{\tr\left(\Ec_{2k,p};\frac{m}{4p}\right)} \Gamma\left(k-\frac{1}{2}, 4\pi \frac{\vt{m}}{4p} (4pv)\right) q^{m} \\
= -\frac{2k-1}{2} (4p)^{k-\frac{1}{2}} \sum_{m < 0} \frac{1}{(4\pi\vt{m})^{k-\frac{1}{2}}} \overline{\tr\left(\Ec_{2k,p};\frac{m}{4p}\right)} \Gamma\left(k-\frac{1}{2}, 4\pi \vt{m} v\right) q^{m}.
\end{multline*}
Thus, we obtain ($v \mapsto 4pv$ in the constant term)
\begin{align*}
\xi_{\frac{3}{2}-k}g_1(\tau) = (-1)^{\frac{k}{2}-1} \frac{\Gamma(k)\zeta(k)}{2^{k-1}\pi^k} (4p)^{k-\frac{1}{2}} \left(k-\frac{1}{2}\right) - (4p)^{k-\frac{1}{2}} \left(k-\frac{1}{2}\right) \sum_{m > 0} \tr\left(\Ec_{2k,p};-\frac{m}{4p}\right) q^{m},
\end{align*}
because $\tr\left(\Ec_{2k,p};-\frac{m}{4p}\right) \in \R$ according to Corollaries \ref{cor:shintanicor} and \ref{cor:Eisensteintraces}.
By virtue of Theorem \ref{thm:CEpreimage}, we also have
\begin{align*}
\xi_{\frac{3}{2}-k} \Fc_{\frac{3}{2}-k,4}^{+}(\tau) &= \frac{2k-1}{3\zeta(1-2k)} \Hs_{k,1,1}(\tau), \\
\xi_{\frac{3}{2}-k} \Fc_{\frac{3}{2}-k,4p}^{+}(\tau) &= \frac{2k-1}{3\zeta(1-2k)} \left(\frac{p-1}{p^{2k}-p} \Hs_{k,1,p}(\tau) + \frac{1}{1-p^{2k-1}} \Hs_{k,p,p}(\tau)\right),
\end{align*}
and thus
\begin{multline*}
\xi_{\frac{3}{2}-k}g_2(\tau) = (-1)^{\frac{k}{2}-1} \frac{(4p)^{k-\frac{1}{2}} (2k-1) \Gamma(k)\zeta(k)}{2^{k} \pi^k \zeta(1-2k)} \\ 
\times \left(p^{-k} \Hs_{k,1,1}(\tau)  + \left(1-p^{-k}\right) \left(\frac{p-1}{p^{2k}-p} \Hs_{k,1,p}(\tau) + \frac{1}{1-p^{2k-1}} \Hs_{k,p,p}(\tau)\right)\right).
\end{multline*}
By virtue of Lemma \ref{lem:Eisensteinfourier} (iii) and Proposition \ref{prop:Millsonfourier}, we infer
\begin{align*}
c_{g_1}^{-}(0) (4pv)^{k-\frac{1}{2}} = c_{g_2}^{-}(0) (4pv)^{k-\frac{1}{2}} = (-1)^{\frac{k}{2}-1} \frac{\Gamma(k)\zeta(k)}{2^{k-1}\pi^k} (4pv)^{k-\frac{1}{2}}.
\end{align*}
Combining with Corollary \ref{cor:Eisensteintraces} shows that the non-holomorphic parts of $g_1$ and $g_2$ coincide. In other words, we obtain
\begin{align*}
\xi_{\frac{3}{2}-k}g_1 = \xi_{\frac{3}{2}-k}g_2,
\end{align*}
and deduce that
\begin{align*}
g_1 - g_2 \eqqcolon g_3 \in \ker\left(\xi_{\frac{3}{2}-k}\right) = M_{\frac{3}{2}-k}^{!}(4p),
\end{align*}
where $g_3$ satisfies the plus space condition. We analyze the growth of $g_3$ at the cusps of $\Gamma_0(4p)$. Let $\af \in \Q$ be a cusp, which is $\Gamma_0(4p)$-inequivalent to $i\infty$. 
\begin{enumerate}
\item By Proposition \ref{prop:Millsonfourier}, the principal part of $g_1$ about $i\infty$ is constant. By Lemma \ref{lem:Eisensteinfourier} (iii), the principal part of $g_2$ about $i\infty$ is constant. Hence, $g_3$ is bounded at $i\infty$.

\item Adapting the proof of \cite{kubota}*{Theorem 2.1.2} slightly to include more general weights, the Eisenstein series $\Fc_{\frac{3}{2}-k,4p}$ is bounded as $\tau \to \af$, so the same is true for its projection to the plus space. Thus, $g_2$ is bounded at $\af$.

\item We justify that $g_1$ is bounded at $\af$. There is no contribution from $c_{g_1}^{-}(0)v^{k-\frac{1}{2}}$, because $k > 0$ and $v \to 0^{+}$ as $\tau \to \af$. Note that $c_{g_1}^{-}(n) = 0$ for $n > 0$. Hence, we need to consider the principal part of $g_1$ about $\af$. Let $\sigma_{\af} \in \slz$ be a scaling matrix for the cusp $\af$. Then, the Fourier expansion of $g_1$ about $\af$ can be read off from
\begin{align*}
g_{\af} \coloneqq g_1 \big\vert_{\frac{3}{2}-k} \sigma_{\af}.
\end{align*}
Since $I^{(\mathrm{M})}\left(\Fc_{2-2k,p}, \tau\right)$ is a vector-valued harmonic Maass form for the Weil representation $\rho_{L^{-}}$ having a constant principal part if and only if $h=0$ by Corollary \ref{cor:cuspparameters} and Proposition \ref{prop:Millsonfourier}, item (1) implies that $g_1$ is bounded at $\af$ as well. See \cite{bor98}*{Example 2.2} for further justification.
\end{enumerate}
Summing up, the function $g_3$ must be holomorphic at the cusps. However, a result of Cohen and Oesterl\'e \cite{coos}*{p.\ C-0-7} asserts that
$
M_{\frac{3}{2}-k}(4p) = \{0\},
$
since $k > 1$. Hence, we arrive at $g_3 = 0$, which completes the proof.
\end{proof}

Lastly, we prove Corollary \ref{cor:millsoncor}.
\begin{proof}[Proof of Corollary \ref{cor:millsoncor}]
Since the raising operator and the slash operator intertwine by \cite{thebook}*{Lemma 5.2}, equation \eqref{eq:Eisensteinraising} yields
\begin{align*}
R_{2-2k}^{k-1} \Fc_{2-2k,N}(\tau) = (k-1)! \Fc_{0,N}(\tau,k).
\end{align*}
Thus, the holomorphic part in Proposition \ref{prop:Millsonfourier} becomes
\begin{multline*}
\sum_{h \pmod*{2p}} \sum_{n > 0} \frac{1}{2\sqrt{n}} \left(\frac{\sqrt{p}}{4\pi\sqrt{n}}\right)^{k-1} \tr\left(\Fc_{0,p}(\cdot,k);-n;h\right) q^{4pn} \\
= \sum_{n > 0} \frac{1}{2\sqrt{\frac{n}{4p}}} \left(\frac{\sqrt{p}}{4\pi\sqrt{\frac{n}{4p}}}\right)^{k-1} \sum_{h \pmod*{2p}} \tr\left(\Fc_{0,p}(\cdot,k);-\frac{n}{4p};h\right) q^{n} \\
= \frac{p^{k-\frac{1}{2}}}{(2\pi)^{k-1}} \sum_{n > 0} n^{-\frac{k}{2}} \sum_{Q \in \Qc_{p,-n} \slash \Gamma_0(p)} \frac{1}{\vt{\left(\Gamma_0(p)\right)_Q}} \Fc_{0,p}\left(\tau_Q,k\right) q^n.
\end{multline*}
Comparing Fourier coefficients to the holomorpic part in Lemma \ref{lem:Eisensteinfourier} (iii) hence yields
\begin{multline*}
\frac{p^{k-\frac{1}{2}}(k-1)!}{(2\pi)^{k-1}} n^{-\frac{k}{2}} \sum_{Q \in \Qc_{p,-n} \slash \Gamma_0(p)} \frac{1}{\vt{\left(\Gamma_0(p)\right)_Q}} \Fc_{0,p}\left(\tau_Q,k\right) \\
= 3 \frac{2}{3} \left(\frac{i}{2}\right)^{k-\frac{3}{2}} \pi (-1)^{\frac{k}{2}-1} \frac{\Gamma(k)\zeta(k)}{2^{k} \pi^k} (4p)^{k-\frac{1}{2}} \\
\times \left(p^{-k} \Ks_{\frac{3}{2}-k,4}^{+}\left(0,n;k+\frac{1}{2}\right)  + \left(1-p^{-k}\right) \Ks_{\frac{3}{2}-k,4p}^{+}\left(0,n;k+\frac{1}{2}\right)\right).
\end{multline*}
Simplifying gives the claim.
\end{proof}

\section{Further questions} \label{sec:questions}

We conclude with a few questions for future work. The first one was suggested by Larry Rolen, and the second one was suggested by Nikolaos Diamantis.
\begin{enumerate}
\item Wagner \cite{wag} defined and investigated $p$-adic analogs of his harmonic preimages of the classical Cohen--Eisenstein series $\Hs_{k,1,1}$ under $\xi_{\frac{3}{2}-k}$. His results are discussed in \cite{thebook}*{Subsection 6.4} too. Do his $p$-adic constructions extend to higher level as well?
\item Diamantis, Lee, Raji and Rolen \cite{DLRR} as well as introduced $L$-series of harmonic Maass forms and reinterpreted the Fourier coefficients of Duke, Imamo\={g}lu and T\'{o}th \cite{dit11annals}. Subsequently, Diamantis and Rolen \cite{DR} studied $L$-values of harmonic Maass forms at integer points extending the work of Bruinier, Funke and Imamo\={g}lu \cite{brufuim}. Their evaluations are consequences of the chosen regularizations respectively. How do these relate to our identities in Corollaries \ref{cor:shintanicor} and \ref{cor:millsoncor}?
\item Can one evaluate the right hand side in Corollary \ref{cor:millsoncor} further? This is motivated by the remarks at the bottom of \cite{thebook}*{p.\ $106$}.
\item Are other regularized theta lifts of Eisenstein series related to Pei and Wang's generalized Cohen--Eisenstein series as well? For example, Bruinier, Ehlen and Yang \cite{BEY} combined a regularized Siegel theta lift with the iterated Maass raising operator from equation \eqref{eq:Maassraising} to access other weights than $0$ and $\frac{1}{2}$.
\end{enumerate}

\end{document}